\documentclass[12pt]{amsart}

\usepackage{amsmath, amsthm, amssymb}
\input xy
\xyoption{all}
\usepackage{mathrsfs}
\usepackage{enumerate}

\usepackage{tikz}
\usepackage[utf8]{inputenc}

\newtheorem{dummy}{}[section]
\newtheorem{theorem}[dummy]{Theorem}
\newtheorem{proposition}[dummy]{Proposition}
\newtheorem{lemma}[dummy]{Lemma}
\newtheorem{corollary}[dummy]{Corollary}
\newtheorem{conjecture}[dummy]{Conjecture}
\theoremstyle{definition}
\newtheorem{definition}[dummy]{Definition}

\newtheorem{remark}[dummy]{Remark}
\newtheorem{example}[dummy]{Example}

\newcommand{\mV}{\overline V}
\newcommand{\mS}{\overline S}
\newcommand{\QQ}{\mathcal Q}
\newcommand{\Cont}{\mathrm Cont}
\newcommand{\PP}{\mathcal P}
\newcommand{\X}{\mathcal X}
\newcommand{\I}{\mathcal I}
\newcommand{\Y}{\mathcal Y}
\newcommand{\Z}{\mathcal Z}

\newcommand{\bP}{\mathbb{P}}

\newcommand{\M}{\overline{\mathcal{M}}}
\newcommand{\cO}{\mathcal{O}}
\newcommand{\bQ}{\mathbb{Q}}
\newcommand{\bC}{\mathbb{C}}

\newcommand{\vir}{\mathrm{vir}}

\newcommand{\Aut}{\mathrm{Aut}}

\newcommand{\ev}{\mathrm{ev}}
\newcommand{\bt}{\mathbf{t}}

\renewcommand{\ev}{\mathrm{ev}}
\newcommand{\q}{\mathfrak{q}}
\renewcommand{\t}{{t}}
\newcommand{\Res}{\mathrm{Res}}

\newcommand{\Tr}{\mathrm{Tr}}
\newcommand{\diag}{\mathrm{diag}}

\newcommand{\red}{\mathrm{red}}
\newcommand{\oQ}{\overline{Q}}

\usepackage{geometry}
\begin{document}

\title[Genus two Gromov--Witten invariants of quintic threefolds]{A Mirror Theorem for Genus Two Gromov--Witten Invariants of Quintic threefolds}

\author[S.~Guo]{Shuai Guo}
\address{School of Mathematical Sciences and Beijing International Center for Mathematical Research, Peking University, Beijing, China }
\email{gs0202@gmail.com}

\author[F.~Janda]{Felix Janda}
\address{Department of Mathematics, University of Michigan, 2074 East Hall, 530 Church Street, Ann Arbor, MI 48109, USA}
\email{janda@umich.edu}

\author[Y.~Ruan]{Yongbin Ruan}
\address{Department of Mathematics, University of Michigan, 2074 East Hall, 530 Church Street, Ann Arbor, MI 48109, USA}
\email{ruan@umich.edu}
\date{September 2017.}

\begin{abstract}
  We derive a closed formula for the generating function of genus two
  Gromov--Witten invariants of quintic 3-folds and verify the
  corresponding mirror symmetry conjecture of Bershadsky, Cecotti,
  Ooguri and Vafa.
\end{abstract}

\maketitle
\tableofcontents

\section{Introduction}

The computation of the Gromov--Witten (GW) theory of compact
Calabi--Yau 3-folds is a central problem in geometry and physics where
mirror symmetry plays a key role.
In the early 90's, the physicists Candelas and his collaborators
\cite{CdOGP} surprised the mathematical community to use the mirror symmetry to derive a conjectural
formula of a certain generating function (the $J$-function, see Section~\ref{sec:twist} for its definition) of genus zero
Gromov--Witten invariants of a quintic 3-fold in terms of  the period
integral or the $I$-function of its B-model mirror.
The effort to prove the formula directly leads to the birth of
mirror symmetry as a mathematical subject.
Its eventual resolution by Givental \cite{Gi96b} and Liu--Lian--Yau
\cite{LLY97} was considered to be a major event in mathematics during
the 90's.
Unfortunately, the computation of higher genus GW invariants of
compact Calabi--Yau manifolds (such as quintic 3-folds) turns out
to be a \emph{very} difficult problem.
For the last twenty years, many techniques have been developed.
These techniques have been very successful for so-called {\em
  semisimple cases such as Fano or toric Calabi--Yau 3-folds.}
In fact, they were understood thoroughly in several different ways.
But these techniques have little effect on our original problem on
compact Calabi--Yau 3-folds.
For example, using B-model techniques, Bershadsky, Cecotti, Ooguri and
Vafa (BCOV) have already proposed a conjectural formula for genus one
and two Gromov--Witten invariants of quintic 3-folds as early as
1993 \cite{BCOV} (see also \cite{YY04}).
It took another ten years for Zinger to prove BCOV's conjecture in
genus one \cite{Zi08}.
During the last ten years, an effort has been made to push Zinger's
technique to higher genus without success.
Nevertheless, the problem inspires many developments in the subject
such as the modularity problem \cite{BCOV, HKQ09}, FJRW-theory \cite{FJR13} and algebraic mathematical GLSM theory \cite{FJR15}.
It was considered as one of guiding problems in the subject.
The main purpose of this article is to prove BCOV's conjecture in
genus two.

To describe the conjecture explicitly, let us consider the so called
$I$-function of the quintic 3-fold
\begin{equation*}
  I(q,z) = z\sum_{d\geq 0} q^d \frac{\prod^{5d}_{k=1}(5H+kz)}{\prod^d_{k=1}(H+kz)^5}
\end{equation*}
where $H$ is a formal variable satisfying $H^4=0$.
The $I$-function satisfies the following Picard--Fuchs equation
\begin{equation*}
  \Big( D_H^4 - 5 q \prod_{k=1}^4 (5D_H+k z) \Big)  I(q,z) = 0,
\end{equation*}
where $D_H := z q\frac{d}{dq} + H$.
We separate $I(q, z)$ into components:
\begin{equation*}
  I(q,z)  = z I_0(q) \mathbf 1 + I_1(q) H  +z^{-1} I_2(q)H^2 + z^{-2} I_3(q) H^3
\end{equation*}

The genus zero mirror symmetry conjecture of quintic 3-fold can be
phrased as a relation between the $J$- and the $I$-function
$$J(Q)=\frac{I(q)}{I_0(q)}$$
up to a mirror map $Q = q e^{\tau_Q(q)}$, where
\begin{equation*}
  \tau_Q(q)  = \frac{I_1(q)}{I_0(q)}.
\end{equation*}
The leading terms of $I_0$ and $\tau_Q$ are
\begin{align*}
  I_0 (q)= \,&1+120\, q+113400 \,q^2+168168000\, q^3+O(q^4),\\
  \tau_Q (q) =\,& 770 \,q+717825 \,q^2+\frac{3225308000}{3} \, q^3+ O(q^4) .
\end{align*}
Now we introduce the following  degree $k$ ``basic'' generators
\begin{equation*}
  \X_k := \frac{d^k}{du^k} \left(\log \frac{I_0} L\right),\quad
  \Y_k:= \frac{d^k}{du^k} \left(\log \frac{I_0  I_{1,1}}{ L^2}\right),\quad
  \Z_k:=  \frac{d^k}{du^k} \left( \log (q^{\frac{1}{5}}L)\right),
\end{equation*}
where $I_{1,1} := 1+ q\frac{d}{dq} \tau_Q$,
$L := (1-5^5q)^{-\frac{1}{5}}$ and $du := L \frac{dq}{q}$.
Some numerical data for these generators are given in Remark~\ref{numer}.

Let $F^{GW}_g(Q)$ be the generating function of genus $g$
Gromov--Witten invariants of a quintic 3-fold.
The following is an equivalent formulation of the genus two mirror
conjecture (Conjecture~\ref{YY}) of \cite{YY04,HKQ09} (see Section~\ref{sec:equivalence} for
the argument).
\begin{conjecture}
  \label{mainconj}
  The genus $2$ GW generating function $F^{GW}_2(Q)$ for the quintic threefold is given by
  \begin{align*}
    F^{GW}_2(Q) = \, & \frac{I^2_0 }{L^2} \cdot \Big(
{\frac {70\,\X_{{3}}}{9}}+{\frac {575\,\X\X_{{2}}}{18}}+\frac{5 \Y\X_{{2}}}{6}+{\frac {557\,{\X}^{3}}{72}}-{\frac {629\,\Y{\X
}^{2}}{72}}-{\frac {23\,{\Y}^{2}\X}{24}}-\frac{{\Y}^{3}}{24}\\&+{
\frac {625\,\Z\X_{{2}}}{36}}-{
\frac {175\,\Z\Y\X}{9}}+{\frac {1441\,\Z_{{2
}}\X}{48}}-{\frac {25\,\Z({\X}^{2}+{\Y}^{2})}{24}}-{\frac {3125\,{\Z}^{2}(\X+\Y)}{288}}\\
&  +{\frac {41\,\Z_{{2}}\Y}{48}}-{\frac {625\,{\Z}^{3}}{144}}+{
\frac {2233\,\Z\Z_{{2}}}{128}}+{\frac {547\,\Z_{{3}}}{72}} \Big),
  \end{align*}
  where $Q = q e^{\tau_Q(q)}$.

  In particular, the leading terms of $F_2^{GW}$ are
  $$
  F^{GW}_2(Q)  = -\frac{5}{144} + \frac{575}{48} Q + \frac{5125}{2} Q^2 +\frac{7930375}{6} Q^3 +O(Q^4).
  $$
\end{conjecture}

The following is our Main Theorem:
\begin{theorem}
  The above genus two mirror conjecture of quintic 3-fold holds.
\end{theorem}

\begin{remark}
  A consequence of above conjecture is that $(L/I_0)^2(q) F^{GW}_2(Q)$
  (more generally, $(L/I_0)^{2g - 2}(q) F^{GW}_g(Q)$) is a homogeneous
  polynomial of the generators $\X_k, \Y_k,\Z_k$.
  We refer to this property as {\em finite generation}.
  On the other hand, the original conjecture (Conjecture~\ref{YY}) is
  an inhomogeneous polynomial of five generators.
  We found it easier to work with a homogeneous polynomial than an
  inhomogeneous polynomial.
  Since the Taylor expansions of the generators are known, we can
  easily compute numerical genus two GW-invariants for any degree.
\end{remark}

As we mentioned previously, it has been ten years since Zinger proved the genus one
BCOV mirror conjecture.
A key new advancement during last ten years was the understanding of global
mirror symmetry which was in the physics literature in the beginning but
somehow lost in its translation into mathematics in the early 90's.
The idea of the global mirror symmetry \cite{CdOGP, BCOV, HKQ09, CR11}
is to view GW theory as a particular limit (large complex structure
limit) of the global B-model theory.
Physicists use the results about the other limits such as the Gepner
limit and conifold limit to yield the computation of the large complex
structure limit/GW theory.
Interestingly, one of the key pieces of information they used is the
regularity of the Gepner limit.
The latter can be interpreted as the existence of FJRW theory.
A natural consequence of the above global mirror symmetry perspective
is a prediction that the GW/FJRW generating functions are quasi-modular forms in some sense and hence are
polynomials of certain finitely many canonical generators \cite{BCOV, HKQ09, YY04}.
This imposes a strong structure for GW theory and we refer it as the
{\em finite generation property}.
For anyone with experience on the complexity of numerical GW
invariants, it is not difficult to appreciate how amazing the finite
generation property is!
In fact, it immediately reduces an infinite computation for all degree
to a finite computation of the coefficients of a polynomial.
Therefore, it should be considered as one of the fundamental problems
in the subject of higher genus GW theory.

The above global mirror symmetry framework was successfully carried
out by the third author and his collaborators \cite{KS, MRS, IMRS} for
certain maximal quotients of Calabi--Yau manifolds.
These examples are interesting in their own right.
Unfortunately, we understand very little about the relation between
the GW-theory of a Calabi--Yau manifold and its quotient.
Hence, the success on its quotient has only a limited impact on our
original problem.
Several years ago, an algebraic-geometric curve-counting theory was
constructed by Fan--Jarvis--Ruan for so called {\em gauged linear
  sigma model (GLSM)} (see \cite{Wi93, Ho13} for its physical origin).
One application of mathematical GLSM theory is to interpret the above
global mirror symmetry as wall-crossing problem for a certain
stability parameter $\epsilon$ of the GLSM-theory.
It leads to a complete new approach to attack the problem without
considering B-model at all.
The first part of the new approach is to vary the stability parameter
from $\epsilon=\infty$ (stable map theory) to $\epsilon=0^+$
(quasimap \cite{CFKM14} or stable quotient \cite{MOP11} theory) and
has been successfully carried out recently by several authors
\cite{CFK, CJR1, CJR2, Z}.
Suppose that $F^{SQ}_{g}(q)$ is the the genus $g$ generating function
of stable quotient theory.
Then, the above authors have proved
\begin{equation*}
  F^{GW}_g(Q)=I_0^{2-2g}(q) F^{SQ}_g(q),
\end{equation*}
which explains the appearance of $I^2_0$ at the right hand side of the
conjecture.
In the current paper, we take the next step to calculate the genus two
generating function in quasimap theory and verify the conjectural
formula in \cite{BCOV, YY04}.

The current paper relies on certain geometric input from \cite{CJR}
(see also \cite{CJRSZ}) which we now describe.
Recall that the virtual cycle of the GLSM (stable map with $p$-field
in this case) moduli space was constructed using cosection
localization on an open moduli space \cite{KiLi13}.
The cosection is not $\bC^*$-invariant which prevents us from applying
the localization technique.
A naive idea is to construct a compactification of the GLSM
moduli space such that the cosection localized class can be identified
with the virtual cycle of the compactified moduli space.
Hopefully, the latter carries $\bC^*$-action and the localization
formula can be applied.
Unfortunately, it is not easy to make naive idea work due to the
difficulty of extending the cosection to the compactified moduli space.
Working with Qile Chen, the last two authors solved the problem by
introducing a certain ``reduced virtual cycle'' on an appropriate log
compactification of the GLSM moduli space.
In a sense, the current article and \cite{CJR} belong together.
Of course, \cite{CJR} provides a general tool with applications beyond
the current article.
We will briefly describe \cite{CJR} (specialized to a quintic 3-fold)
in Section~\ref{sec:geo}.

Taking the localization formula of \cite{CJR} as an input, we can
express genus $g$ Gromov--Witten invariants of a quintic 3-fold as a
graph sum of twisted equivariant Gromov--Witten invariants of $\bP^4$
and certain {\em effective invariants}.
When $g=2, 3$, the effective invariants can be computed from known
degree zero Gromov--Witten invariants.
The main content of the current article is to extract a closed formula
for the generating functions.
This is of course difficult in general.
We solve it using Givental formalism.
A subtle and yet interesting phenomenon is the choice of twisted
theory.
The general twisted theory naturally depends on six equivariant
parameters, five for the base $\bP^4$ and one for the twist.
It is complicated to study the general twisted theory, and therefore
Zagier--Zinger \cite{ZZ08} specialize the equivariant parameters of
the base to scalar multiples of $(1, \xi, \xi^2, \xi^3, \xi^4, 0)$,
where $\xi$ is a primitive fifth root of unity.
Under this specialization, they show that the twisted theory is
generated by the five generators predicted by physicists.
Unfortunately, in our work we cannot set the equivariant parameter for
the twist to zero.
As a consequence, we have to introduce four \emph{extra generators}.
It was a miracle to us that the terms involving the four extra
generators cancel and we have our theorem!
The appearance of four extra generators has a direct impact to our
future work for $g \geq 3$.
For example, while no additional geometric input is needed to apply
our method to genus $3$ to prove the conjectural formula of
Klemm-Katz-Vafa \cite{KKV99}, and we could proceed with the methods
developed in this paper by brutal force, the resulting proof would not
be very illuminating.
Recall that there is a conjectural formula up to genus 51 by A.~Klemm
and his collaborators \cite{HKQ09}.
Our eventual goal is to reach genus 51 and go beyond.
To do so, we have to understand better the cancellation of terms
involving extra generators.
We will leave this to a future research \cite{GJR}.

The paper is organized as follows.
In Section~\ref{sec:geo}, we will summarize the relevant compactified
moduli space and its localization formula from \cite{CJR}.
The detailed analysis of contributions of the localization graphs and
their closed formulae in terms of generators are stated in
Section~\ref{sec:twist}.
The main theorem directly follows from these closed formulae.
In Section~\ref{sec:structure}, we derive important results about the
twisted theory, which we then apply in
Section~\ref{proofofproposition} to yield a proof of the closed
formulae.
Finally, we would like to mention several independent approaches to
higher genus problem by Maulik and Pandharipande \cite{MaPa06},
Chang--Li--Li--Liu \cite{CLLL} and Guo--Ross \cite{GuRo,GuRo2}.

The last two authors would like to thank Qile Chen for the collaboration
which provides the geometric input to the current work.
The first author was partially supported by the NSFC grants 11431001 and 11501013.
The second author was partially supported by a Simons Travel Grant.
The third author was partially
supported by NSF grant DMS 1405245 and NSF FRG grant DMS 1159265.

\section{A localization formula}
\label{sec:geo}

The main geometric input is a formula computing GW-invariants of
a quintic 3-fold in terms of a twisted theory and a certain
``effective'' theory.
This formula is obtained by localization on a compactified moduli
space of stable maps with a $p$-field.
The proof of the formula and details about the moduli space can be
found in \cite{CJR} (see also \cite{CJRSZ}).

In Section~\ref{sec:geo:mod}, we give an overview of the definition of
the relevant moduli spaces, and in Section~\ref{sec:geo:loc}, we
explain the localization formula in the general case.
We then, in Section~\ref{sec:geo:g2}, specialize to genus two.
Finally, in Section~\ref{sec:geo:ex}, we illustrate the formula by
directly computing the genus two, degree one invariant.

\subsection{Moduli space}
\label{sec:geo:mod}

Let $\M_{g, 0}(Q_5, d)$ denote the moduli space of genus-$g$,
unpointed, degree-$d$ stable maps $f\colon C \to Q_5$ to the quintic
threefold $Q_5 \subset \bP^4$.
This moduli space admits a perfect obstruction theory and hence a
virtual class $[\M_{g, 0}(Q_5, d)]^\vir$ of virtual dimension zero,
whose degree is defined to be the Gromov--Witten invariant $N_{g, d}$.

For the computations in this paper, it will be much more convenient to
work with the moduli space $\oQ_{g, 0}(Q_5, d)$ of genus-$g$,
unpointed, degree-$d$ stable quotients (or quasimaps) to the quintic
threefold $Q_5 \subset \bP^4$ instead.
We refer to \cite{MOP11, CFKM14} for the definition of this moduli
space, mentioning here just that it parameterizes prestable curves $C$
together with a line bundle $L$ and sections $s \in H^0(L^{\oplus 5})$
satisfying the equation of $Q_5$ and a stability condition.
The moduli space $\oQ_{g, 0}(Q_5, d)$ also admits a perfect
obstruction theory and virtual class $[\oQ_{g, 0}(Q_5, d)]^\vir$,
which after integration defines the stable quotient invariant
$N_{g, d}^{SQ}$.
By the wall-crossing formula \cite{CFK, CJR1}, the information of the
$N_{g, d}^{SQ}$ is equivalent to the information of the $N_{g, d}$.
In genus $g \ge 2$, the wall-crossing formula says that, if
\begin{equation*}
  F_g^{GW}(q) := \sum_{d = 0}^\infty q^d N_{g, d}, \qquad F_g^{SQ}(q) := \sum_{d = 0}^\infty q^d N_{g, d}^{SQ}
\end{equation*}
are the generating series of stable map and stable quotient
invariants, then
\begin{equation*}
  F_g^{GW}(Q) = I_0^{2g - 2} F_g^{SQ}(q).
\end{equation*}

It is difficult to directly access the stable quotient (or stable map)
theory of $Q_5$, in part because $Q_5$ is a ``non-linear'' object.
By results of Chang--Li \cite{ChLi12}\footnote{To be precise, we use
  the result of Chang--Li to rewrite Gromov--Witten invariants of the
  quintic in terms of stable maps with $p$-fields.
  After that, we apply the wall-crossing \cite{CJR2} to move to stable
  quotients with a $p$-field.}, we can instead work on the more linear
moduli space $\oQ_{g, 0}(\bP^4, d)^p$ of stable quotients with a
$p$-field, that is the cone
$\pi_* (\omega_\pi \otimes \mathcal L^{- \otimes 5})$ over
$\oQ_{g, 0}(\bP^4, d)$.
Here, $\pi$ denotes the universal curve and $\mathcal L$ denotes the
universal line bundle $L$.
The moduli space $\oQ_{g, 0}(\bP^4, d)^p$ also has a perfect
obstruction theory and hence a virtual class.
However, because $\oQ_{g, 0}(\bP^4, d)^p$ is in general not compact,
this virtual class cannot directly be used to define invariants.
To circumvent this problem, Chang--Li introduce a cosection $\sigma$
of the obstruction sheaf, and show that
\begin{equation*}
  N_{g, d}^{SQ}
  = \int_{[\oQ_{g, 0}(\bP^4, d)^p]^\vir_\sigma} (-1)^{1 - g + 5d},
\end{equation*}
where $[\oQ_{g, 0}(\bP^4, d)^p]^\vir_\sigma$ is the cosection
localized virtual class of $\oQ_{g, 0}(\bP^4, d)^p$, which is
supported on the compact moduli space $\oQ_{g, 0}(Q_5, d)$.

The main new idea of \cite{CJR} is to define a modular
compactification $X_{g, d}$ of $\oQ_{g, 0}(\bP^4, d)^p$ with a perfect
obstruction theory extending the one of $\oQ_{g, 0}(\bP^4, d)^p$ such
that the cosection $\sigma$ extends without acquiring additional
degeneracy loci.
It is then easy to see that
\begin{equation*}
  \int_{[\oQ_{g, 0}(\bP^4, d)^p]^\vir_\sigma} (-1)^{1 - g + 5d}
  = \int_{[X_{g, d}]^\vir} (-1)^{1 - g + 5d}.
\end{equation*}
If $X_{g, d}$ furthermore admits a non-trivial torus action (and the
perfect obstruction theory is equivariant), we can apply virtual
localization to the right hand side to express it in terms of
hopefully simpler fixed loci.

A suitable compactification $X_{g, d}$ is given by a space of stable
quotients $(C, L, s)$ together with a log-section
$\eta\colon C \to \bP := \bP(\omega_C \otimes L^{-\otimes 5} \oplus
\cO_C)$ where the target is equipped with the divisorial
log-structure corresponding to the infinity section.
The open locus where the log-section $\eta$ does not touch the
infinity section is clearly isomorphic to $\oQ_{g, 0}(\bP^4, d)^p$,
and the canonical perfect obstruction theory of $X_{g, n, d}$ extends
the one of $\oQ_{g, 0}(\bP^4, d)^p$.
Unfortunately, the cosection $\sigma$ becomes singular on the
complement of $\oQ_{g, 0}(\bP^4, d)^p$, Still, it can be extended to a
homomorphism to a non-trivial line bundle $L_N^\vee$, which comes with
a canonical section $\cO \to L_N^\vee$, and we can define a new
(\emph{reduced}) perfect obstruction theory by ``removing'' (taking
the cone) under the induced map from the obstruction theory to the
complex $[\cO \to L_N^\vee]$.
With this reduced perfect obstruction theory, $X_{g, d}$ satisfies all
desired properties.

\begin{remark}
  The construction of the modular compactification and its reduced
  perfect obstruction theory can also be carried out in the setting of
  stable maps, and it satisfies all of the analogous properties.
\end{remark}

\subsection{Localization}
\label{sec:geo:loc}

We now consider virtual localization \cite{GrPa99} of the reduced
perfect obstruction theory of $X_{g, d}$ with respect to the
$\mathbb C^*$-action of $X_{g, d}$ that scales $\eta$.
Let $t$ be the corresponding equivariant parameter.
The fixed loci of the $\mathbb C^*$-action are indexed by bivalent
graphs $\Gamma$ with $n$ legs.
We let $V(\Gamma)$, $E(\Gamma)$ be the corresponding sets of vertices
and edges.
The vertices $v \in V(\Gamma)$ correspond to components (or isolated
points) of the curve $C$ sent to either the zero or infinity section.
We define the bivalent structure
$V(\Gamma) = V_0(\Gamma) \sqcup V_\infty(\Gamma)$ accordingly.
Furthermore, let $g(v)$ (respectively, $d(v)$) be the genus of such a
component (respectively, the degree of $L$ on this component).
A vertex $v$ is \emph{unstable} if and only if it corresponds to an
isolated point of $C$, that is when $g(v) = 0$ and either $n(v) = 1$
or $n(v) = 2$ and $d(v) = 0$.
The edges $e \in E(\Gamma)$ correspond to the remaining components of
$C$, which are rational, each contracted to a point in $\bP^4$, and
mapped via a degree-$\delta(e)$ Galois cover to the corresponding
component of $\bP$ (with a possible base point at the zero section).

The fixed locus corresponding to a dual graph $\Gamma$ is (up to a
finite map) isomorphic to
\begin{equation*}
  M_\Gamma := \prod_{v \in V_0(\Gamma)} \oQ_{g(v), n(v)}(\bP^4, d(v)) \times_{(\bP^4)^{|E(\Gamma)|}} \prod_{v \in V_\infty(\Gamma)} \oQ_{g(v), n(v)}(\bP, d(v), \mu(v))^\sim,
\end{equation*}
where unstable moduli spaces $\oQ_{g(v), n(v)}(\bP^4, d(v))$ are
defined to be a copy of $\bP^4$, and the rubber moduli space
$\oQ_{g(v), n(v)}(\bP, d(v), \mu(v))^\sim$ generically parameterizes
stable quotients $(C, L, s)$ together with a nonzero holomorphic
section of $\omega \otimes L^{-\otimes 5}$ up to scaling with zeros
prescribed by $\mu(v)$.
Here, $\mu(v)$ is the $n(v)$-tuple of integers consisting of 0 for
each leg, and $\delta(e) - 1$ for each edge $e$ at $v$.
In analogy with \cite{Po06}, we will also refer to the rubber moduli
space as the ``effective'' moduli space.

In order for a fixed locus corresponding to a graph to be non-empty,
there are many constraints on the decorated dual graph.
First, we must have
\begin{equation*}
  g = \sum_{v \in V(\Gamma)} g(v) + h^1(\Gamma), \qquad d = \sum_{v \in V(\Gamma)} d(v).
\end{equation*}
Second, there is a stability condition which says that for any vertex
$v \in V_0(\Gamma)$ of genus zero and valence one, the corresponding
unique edge $e$ needs to satisfy $\delta(e) > 1 + 5d(v)$.
Third, for every $v \in V_\infty(\Gamma)$, the partition $\mu(v)$ must
have size $2g(v) - 2 - 5d(v) \ge 0$.
Note that the third condition implies that $g(v) \ge 1$ for each
$v \in V_\infty(\Gamma)$.
A localization graph satisfying all these conditions is depicted in
Figure~\ref{fig:glocgraph}.

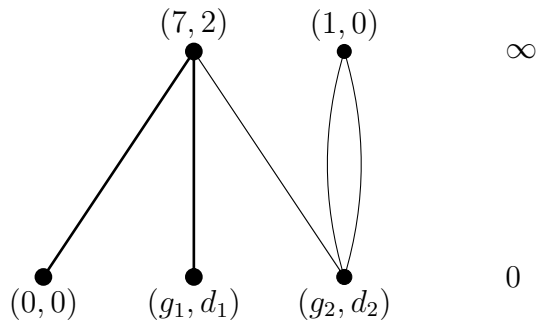
\begin{figure}
  \centering
  \begin{tikzpicture}
    \draw[fill,line width=1pt] (0, 0) circle(1mm) node[below] {$(0, 0)$} -- (2, 3);
    \draw[fill,line width=1pt] (2, 0) circle(1mm) node[below] {$(g_1, d_1)$} -- (2, 3) circle(1mm) node[above] {$(7, 2)$};
    \draw[fill] (2, 3) -- (4, 0) circle(1mm) node[below] {$(g_2, d_2)$};
    \draw (4, 0) .. controls (3.7, 1) and (3.7, 2) .. (4, 3);
    \fill (4, 3) circle(1mm) node[above] {$(1, 0)$};
    \draw (4, 0) .. controls (4.3, 1) and (4.3, 2) .. (4, 3);
    \draw (6, 0) node[right] {$0$};
    \draw (6, 3) node[right] {$\infty$};
  \end{tikzpicture}
  \caption{A localization graph in genus $g_1 + g_2 + 9$ and degree
    $d_1 + d_2 + 2$.
    The bottom vertices lie in $V_0$, the top vertices lie in
    $V_\infty$.
    The pair $(g(v), d(v))$ is specified at each vertex.
    An edge $e$ is thick if $\delta(e) = 2$.
    Otherwise, $\delta(e) = 1$.}
  \label{fig:glocgraph}
\end{figure}

The contribution of a decorated graph $\Gamma$ can then be written as
\begin{multline}
  \label{eq:localization}
  \frac 1{|\Aut|} \int_{M_\Gamma} \Delta^! \left(
  \prod_{v \in V_0(\Gamma)} \frac{e(-R\pi_{v, *}(\omega_{\pi_v} \otimes \mathcal L^{-\otimes 5}) \otimes [1])}{\prod_{e\text{ at }v} (t - 5\ev_e^*(H))(\frac{t - 5\ev_e^*(H)}{\delta(e)} - \psi_e)} \cap [\oQ_{g(v), n(v)}(\bP^4, d(v))]^\vir\right. \\
  \left.\times \prod_{v \in V_\infty(\Gamma)} \frac t{-t + 5H - \psi_0} \cap [\oQ_{g(v), n(v)}(\bP, d(v), \mu(v))^\sim]^\red
  \times \prod_{e \in E(\Gamma)} \frac 1{\delta(e) \prod_{i = 1}^{\delta(e) - 1} \frac{t - 5\ev_e^*(H)}{\delta(e)}}\right),
\end{multline}
with the notation:
\begin{itemize}
\item
  $\Delta\colon (\bP^4)^{|E(\Gamma)|} \to (\bP^4)^{|E(\Gamma)|} \times
  (\bP^4)^{|E(\Gamma)|}$: diagonal map
\item $\pi_v$: universal curve corresponding to $v \in V_0(\Gamma)$
\item $\ev_e^*(H)$: pullback of $H$ via any of the two evaluation maps
  corresponding to $e$
\item $[1]$: line bundle with Chern class $t$
\item $5H - \psi_0$: a universal divisor class on the effective moduli space
\item
  $[\oQ_{g(v), n(v)}(\bP, d(v),
  \mu(v))^\sim]^\red$: a reduced virtual class, which is
  discussed below
\end{itemize}
In the unstable case that $v \in V_0(\Gamma)$ has valence $2$ and is
connected to two edges $e_1$ and $e_2$, we define
\begin{equation}
  \label{eq:unstable}
  \frac{e(-R\pi_{v, *}(\omega_{\pi_v} \otimes \mathcal L^{-\otimes 5}) \otimes [1])}{\prod_{e\text{ at }v} (t - 5\ev_e^*(H))(\frac{t - 5\ev_e^*(H)}{\delta(e)} - \psi_e)} \cap [\oQ_{g(v), n(v)}(\bP^4, d(v))]^\vir
\end{equation}
to be
\begin{equation*}
  \frac 1{(t - 5H)\left(\frac{t - 5H}{\delta(e_1)} + \frac{t - 5H}{\delta(e_2)}\right)},
\end{equation*}
and when $v$ has genus zero and is connected to a single edge $e$, we
define \eqref{eq:unstable} as
\begin{equation}
  \label{eq:unstable01}
  \frac{(t - 5H)^{5d(v) + 1} (5d(v) + 1)!}{\delta(e)^{5d(v) + 1}}.
\end{equation}

Most parts of \eqref{eq:localization} are effectively computable, such
as the integrals over the moduli spaces for $v \in V_0(\Gamma)$, which
are twisted \cite{CoGi07} invariants of $\bP^4$.
The most difficult part of the formula is related to the effective
moduli spaces.
Fortunately, these integrals are highly constrained.

The reduced virtual class
\begin{equation*}
  [\oQ_{g, n}(\bP, d, \mu)^\sim]^\red
\end{equation*}
has dimension
\begin{equation*}
  n - (2g - 2 - 5d),
\end{equation*}
which is one more than the naive virtual dimension.
Together with pullback properties of the reduced virtual class, this
implies that the virtual class vanishes unless all parts of $\mu$ are
0 or 1.
Furthermore, there is a dilaton equation which implies that
\begin{equation*}
  \int_{[\oQ_{g, n}(\bP, d, \mu)^\sim]^\red} \psi_0^{n - 2g + 2}
  = \frac{(2g - 2 + n)!}{(4g - 4 - 5d)!} \int_{[\oQ_{g, 2g - 2}(\bP, d, (1, \dotsc, 1))^\sim]^\red} 1
\end{equation*}
as long as $g \ge 2$.
When $g = 1$, we need to have $d = 0$, and
$\oQ_{1, n}(\bP, 0, (0, \dotsc, 0))^\sim \cong \oQ_{1, n} \times
\bP^4$.
The virtual class is also explicit, and given by
\begin{equation*}
  [\oQ_{1, n}(\bP, 0, (0, \dotsc, 0))^\sim]^\red
  = e((T_{\bP^4} + \cO - \cO_{\bP^4}(5)) \otimes \mathbb E^\vee)
  = 205 H^4 + 40 H^3 \lambda_1,
\end{equation*}
where $\mathbb E$ denotes the Hodge (line) bundle, and $\lambda_1$ its
first Chern class.
There are also explicit formulae for $5H - \psi_0$ in this case, for
instance, when $n = 1$, we have
$5H - \psi_0 = 5H - \lambda_1 = 5H - \psi_1$, and, when $n = 2$, we
have $5H - \psi_0 = 5H - \psi_1 = 5H - \psi_2$.

All in all, \eqref{eq:localization} gives an explicit computation of
any Gromov--Witten invariant of the quintic, up to the determination
of the constants
\begin{equation*}
  c_{g, d} := \int_{[\oQ_{g, 2g - 2}(\bP, d, (1, \dotsc, 1))]^\red} 1\in \mathbb Q,
\end{equation*}
which are defined for every $g \ge 2, d \ge 0$ such that
$d \le \frac{2g - 2}5$.
We call these constants ``effective invariants''.
Note that for any particular genus, only finitely many of these
invariants are needed.

\begin{remark}
  All of the discussion of this section can also be carried out in the
  stable maps setting, with only the following essential differences:
  Vertices $v \in V_0(\Gamma)$ with $(g(v), n(v)) = (0, 1)$ are stable
  unless $d(v) = 0$.
  Therefore, for such vertices $v$, the corresponding edge $e$ only
  needs to have $\delta(e) > 1$.
  Accordingly, we also need to replace $5d(v) + 1$ by $1$ in
  \eqref{eq:unstable01}.

  Because of this, there are many more localization graphs in the
  stable maps setting than in the stable quotient setting.
  In fact, while in any case, there are only finitely many types of
  localization graphs for any fixed $g$ and $d$, in the stable
  quotient setting, for fixed $g$, the number of localization graphs
  does not depend on $d$ (as long as $d$ is large enough).
  This is the main technical advantage of the stable quotient theory
  for our purpose.
\end{remark}

\subsection{Formula in genus two}
\label{sec:geo:g2}

\begin{figure}
  \centering
  \begin{tikzpicture}
    \fill (0, 0) circle(1mm) node[below] {$(2, d)$};
    \draw[fill] (2, 0) circle(1mm) node[below] {$(1, d)$} -- (2, 3) circle(1mm) node[above] {$(1, 0)$};
    \draw[fill] (4, 3) node[above] {$(1, 0)$} circle(1mm) -- (5, 0) circle(1mm) node[below] {$(0, d)$} -- (6, 3) circle(1mm) node[above] {$(1, 0)$};
    \draw (8, 0) .. controls (7.7, 1) and (7.7, 2) .. (8, 3);
    \draw (8, 0) .. controls (8.3, 1) and (8.3, 2) .. (8, 3);
    \fill (8, 0) circle(1mm) node[below] {$(0, d)$};
    \fill (8, 3) circle(1mm) node[above] {$(1, 0)$};
    \fill (11, 3) circle(1mm) node[above] {$(2, 0)$};
    \fill (10, 0) node[below] {$(0, 0)$} circle(1mm);
    \fill (12, 0) node[below] {$(0, 0)$} circle(1mm);
    \draw[line width=1pt] (10, 0) -- (11, 3) -- (12, 0);
    \draw (0,-1) node[below] {$\Gamma^2$};
    \draw (2,-1) node[below] {$\Gamma^1$};
    \draw (5,-1) node[below] {$\Gamma^0_b$};
    \draw (8,-1) node[below] {$\Gamma^0_a$};
    \draw (13, 0) node[right] {$0$};
    \draw (13, 3) node[right] {$\infty$};
  \end{tikzpicture}
  \caption{The localization graphs in genus two.}
  \label{fig:g2graph}
\end{figure}
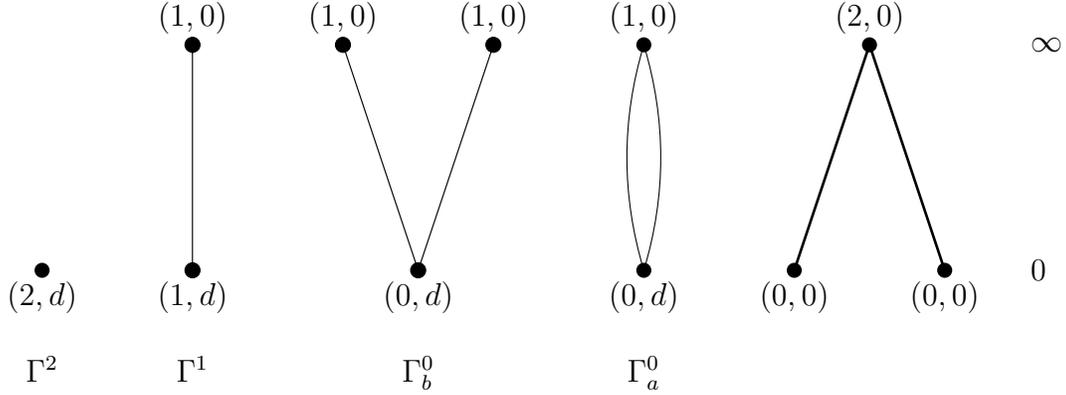

We now specialize the localization formula to genus two, and apply it
to the computation of the generating series
\begin{equation*}
  F_2^{SQ}(q)
  = \sum_{d = 0}^\infty (-1)^{1 - g + 5d} q^d \int_{[X_{2, d}]^\vir} 1.
\end{equation*}
There are 5 localization graphs, which are shown in
Figure~\ref{fig:g2graph}.
Note that the fifth graph can only occur when $d = 0$, and that its
contribution is given by the constant $-c_{2, 0}$.
We label the first four graphs by $\Gamma^2$, $\Gamma^1$, $\Gamma^0_b$
and $\Gamma^0_a$, respectively.

We introduce the bracket notation
\begin{multline*}
  \langle\alpha_1, \dotsc, \alpha_n\rangle_{g, n}^{t,SQ}
  = \sum_{d = 0}^\infty q^d \int_{[\oQ_{g, n}(\bP^4, d)]^\vir} e(R\pi_* \mathcal L^{\otimes 5} \otimes [1]) \prod_{i = 1}^n \ev_i^*(\alpha_i) \\
  = \sum_{d = 0}^\infty (-1)^{1 - g + 5d} \int_{[\oQ_{g, n}(\bP^4, d)]^\vir} e(-R\pi_* (\omega \otimes \mathcal L^{\otimes -5}) \otimes [1]) \prod_{i = 1}^n \ev_i^*(\alpha_i).
\end{multline*}

The contribution of $\Gamma^2$ is then simply given by
\begin{equation*}
  \Cont_{\Gamma^{2}} :=
  \langle\rangle_{2,0}^{t,SQ}.
\end{equation*}

For the contribution of $\Gamma^1$, we first need the computation
\begin{equation*}
  \int_{\M_{1, 1} \times \bP^4} \frac{t}{-t + 5H - \psi_0} (205 H^4 + 40 H^3 \lambda_1)
  = -\frac 53 H^3 + \frac 5{24} H^4 t^{-1}.
\end{equation*}
Therefore,
\begin{equation*}
  \Cont_{\Gamma^{1}} := -\left\langle \frac{-\frac{5}{3}H^3+\frac{5}{24}H^4 t^{-1}}{(t-5H)(t-5H-\psi)} \right\rangle_{1,1}^{t, SQ},
\end{equation*}
and we can also directly compute:
\begin{equation*}
  \Cont_{\Gamma_{b}^{0}}:= \left\langle \frac{-\frac{5}{3}H^3+\frac{5}{24}H^4 t^{-1}}{(t-5H)(t-5H-\psi_1)},
\frac{-\frac{5}{3}H^3+\frac{5}{24}H^4 t^{-1}}{(t-5H)(t-5H-\psi_2)} \right\rangle_{0,2}^{t, SQ}
\end{equation*}

Finally, for the contribution of $\Gamma_a^0$, we need to know
\begin{equation*}
  \int_{\M_{1, 2} \times \bP^4} \frac{t}{-t + 5H - \psi_0} (205 H^4 + 40 H^3 \lambda_1)
  = \frac 53 H^3 t^{-1} + \frac{65}8 H^4 t^{-2}.
\end{equation*}
Thus,
\begin{equation*}
  \Cont_{\Gamma_{a}^{0}} :=  \left\langle  \frac{\frac{5}{3} (H^3 \otimes H^4 + H^4 \otimes H^3) t^{-1} + \frac{65}{8} H^4 \otimes H^4 t^{-2} }{(t-5H)(t-5H-\psi_1) \, (t-5H)(t-5H-\psi_2)} \right\rangle_{0,2}^{t, SQ}.
\end{equation*}

Summing all contributions gives
\begin{equation}
  \label{eq:explicitform}
  F_2^{SQ}(q) = \Cont_{\Gamma^{2}} + \Cont_{\Gamma^{1}} + \frac 12 \Cont_{\Gamma_{a}^{0}} + \frac 12 \Cont_{\Gamma_{b}^{0}} - c_{2, 0}.
\end{equation}

\subsection{Example}
\label{sec:geo:ex}

To illustrate our genus two formula, we compute explicitly the degree
one invariant, that is the coefficient of $q^1$ in $F_2^{SQ}$ defined
by \eqref{eq:explicitform}, and match it with the value given by
Conjecture~\ref{mainconj}, that is the $q^1$-coefficient of
\begin{equation*}
  F_2^{SQ}(q) = (I_0(q))^{-2} F_2^{GW}(Q)
  = -\frac 5{144} + \frac{325}{16} q + \frac{366875}{24} q^2 + \frac{1030721125}{48} q^3 + O(q^4).
\end{equation*}
Note that the value of the constant $c_{2, 0}$ is irrelevant for this
computation.
To compute the $q^1$-coefficients of the other terms on the right hand
side of \eqref{eq:explicitform}, we use localization for an additional
diagonal $(\bC^*)^5$-action on the base $\bP^4$ with corresponding
localization parameters $\lambda_i$.
We refer to \cite[Section~7]{MOP11} and \cite{GrPa99} for the
enumeration of fixed points and identification of localization
contributions that we use below.
It will be convenient for the computation to assume that
$\lambda_i = \xi^i \lambda$ where $\xi$ is a primitive fifth root of
unity.
We index the fixed points of $\bP^4$ by $i \in \{0, \dotsc, 4\}$.
Note that the Euler class of the Poincaré dual of fixed point $i$ is
given by $\prod_{j \neq i} (\lambda_i - \lambda_j)$.

\subsubsection{Genus two contribution}

We begin with the computation of $\Cont_{\Gamma^{2}}$ via
localization.
First, note that there are two types of fixed loci for localization on
$\oQ_{2, 0}(\bP^4, 1)$.

The first type of fixed locus is where, except for an order one base
point, the entire source curve is contracted to a fixed point $i$.
Such a fixed locus is isomorphic to $\M_{2, 1}$.
The contribution of the virtual class of this fixed locus is
\begin{equation}
  \label{eq:locg2vir}
  \frac{\prod_{j \neq i} ((\lambda_i - \lambda_j)^2 - \lambda_1 (\lambda_i - \lambda_j) + \lambda_2)}{\prod_{j \neq i} (\lambda_i - \lambda_j)(\lambda_i - \lambda_j - \psi_1)},
\end{equation}
where, by abuse of notation, $\lambda_1$ and $\lambda_2$ denote the
Chern classes of the Hodge bundle.
The contribution of the twisting by
$e(R\pi_* \mathcal L^{\otimes 5} \otimes [1])$ is
\begin{equation}
  \label{eq:locg2twist}
  \frac{(t + 5\lambda_i - 5\psi_1)\dotsb(t + 5\lambda_i - \psi_1)(t + 5\lambda_i)}{(t + 5\lambda_i)^2 - \lambda_1 (t + 5\lambda_i) + \lambda_2}.
\end{equation}
Summing the product of \eqref{eq:locg2vir} and \eqref{eq:locg2twist}
over all five fixed loci, taking the coefficient of $t^0$ and
integrating, gives the total resulting contribution of these 5 fixed
loci
\begin{multline*}
  \int_{\M_{2, 1}} 1370 \psi_1^4 - 3075 \psi_1^2 \lambda_1^2 + 2100 \psi_1 \lambda_1^3 + 75 \psi_1^2 \lambda_2 + 2925 \psi_1 \lambda_1 \lambda_2 \\
  = 1370 \cdot \frac 1{1152} - 3075 \cdot \frac 7{2880} + 2100 \cdot \frac 1{1440} + 75 \cdot \frac 7{5760} + 2925 \cdot \frac 1{2880}
  = -\frac{4285}{1152},
\end{multline*}
where we used a few well-known intersection numbers on $\M_{2, 1}$.

The second type of fixed locus corresponds to the locus of two genus
one components contracted to fixed points $i \neq j$, and which are
connected by a degree one cover of the torus fixed curve connecting
$i$ and $j$.
This fixed locus is isomorphic to $\M_{1, 1} \times \M_{1, 1}$ (up to
a $\mathbb Z/2\mathbb Z$-automorphism group that we will address
later).
The contribution of the virtual class to this fixed locus is
\begin{equation*}
  \prod_{k \neq i} \frac{\lambda_i - \lambda_k - \lambda_{1a}}{\lambda_i - \lambda_k} \prod_{k \neq j} \frac{\lambda_j - \lambda_k - \lambda_{1b}}{\lambda_j - \lambda_k}
  \frac 1{(\lambda_i - \lambda_j - \psi_{1a})(\lambda_j - \lambda_i - \psi_{1b})},
\end{equation*}
where $\psi_{1a}$, $\psi_{1b}$, $\lambda_{1a}$ and $\lambda_{1b}$
denote the cotangent and Hodge classes on each $\M_{1, 1}$-factor.
The contribution of the twisting is
\begin{equation*}
  \frac{(t + 5\lambda_i) \dotsb (t + 5\lambda_j)}{(t + 5\lambda_i - \lambda_{1a}) (t + 5\lambda_j - \lambda_{1b})}.
\end{equation*}
Summing up the contribution from all of the 20 fixed loci gives
\begin{equation*}
  \frac{2965}{288}.
\end{equation*}

\subsubsection{Genus one contribution}

We now compute $\Cont_{\Gamma^{1}}$.
There are also two types of fixed loci for $\oQ_{1, 1}(\bP^4, 1)$.

The first type of fixed locus is where, except for an order one base
point, the entire source curve is contracted to a fixed point $i$.
Such a fixed locus is isomorphic to $\M_{1, 2}$.
The contribution of the virtual class of this fixed locus is
\begin{equation*}
  \prod_{j \neq i} \frac{\lambda_i - \lambda_j - \lambda_1}{(\lambda_i - \lambda_j) (\lambda_i - \lambda_j - \psi_2)},
\end{equation*}
and the contribution of the twisting is
\begin{equation*}
  \frac{(t + 5\lambda_i - 5\psi_2) \dotsb (t + 5 \lambda_i - \psi_2)(t + 5\lambda_i)}{t + 5\lambda_i - \lambda_1}.
\end{equation*}
In addition, we need to consider the insertion.
For this, note that $\ev_1^*(H) = \lambda_i$, and that the descendent
$\psi$-class is $\psi_1$.
Thus, the insertion gives a factor of
\begin{equation*}
  \frac{-\frac 53 \lambda_i^3 + \frac 5{24} \lambda_i^4 t^{-1}}{(t - 5\lambda_i)(t - 5\lambda_i - \psi_1)}.
\end{equation*}
Summing over $i$ gives the contribution of
\begin{equation*}
  \frac{975}{64}.
\end{equation*}

The second type of fixed locus is where there is a genus one curve
contracted to a fixed point $i$ which is connected via a node to a
rational component mapping isomorphically to the fixed line connecting
fixed point $i$ to another fixed point $j$ such that the preimage of
fixed point $j$ is the marking.
This locus is isomorphic to $\M_{1, 1}$.
The contribution of the virtual class is
\begin{equation*}
  \frac{\prod_{k \neq i} (\lambda_i - \lambda_k - \lambda_1)}{\prod_{k \neq i} (\lambda_i - \lambda_k) \prod_{k \neq j} (\lambda_j - \lambda_k)}
  \frac 1{\lambda_i - \lambda_j - \psi_1},
\end{equation*}
and the contribution of the twisting is
\begin{equation*}
  \frac{(t + 5\lambda_i) (t + 4\lambda_i + \lambda_j) \dotsb (t + 5\lambda_j)}{t + 5\lambda_i - \lambda_1}.
\end{equation*}
Now, $\ev_1^*(H) = \lambda_j$, and the descendent $\psi$-class is
given by $\lambda_i - \lambda_j$.
So, the insertion gives a factor of
\begin{equation*}
  \frac{-\frac 53 \lambda_j^3 + \frac 5{24} \lambda_j^4 t^{-1}}{(t - \lambda_j)(t - 4\lambda_j - \lambda_i)}.
\end{equation*}
Summing over all $i \neq j$ gives
\begin{equation*}
  -\frac{3425}{288}.
\end{equation*}

\subsubsection{Genus zero contributions}

We finally compute $\Cont_{\Gamma^{0}_a}+\Cont_{\Gamma^{0}_b}$.
Note that we can combine the two insertions:
\begin{multline*}
  \left(-\frac 53 H^3 + \frac 5{24} H^4 t^{-1}\right)^{\otimes 2} + \frac 53 (H^3 \otimes H^4 + H^4 \otimes H^3) t^{-1} + \frac{65}8 H^4 \otimes H^4 t^{-2} \\
  = \frac{25}9 H^3 \otimes H^3 + \frac{95}{72} (H^3 \otimes H^4 + H^4 \otimes H^3) t^{-1} + \frac{4705}{576} H^4 \otimes H^4 t^{-2}
\end{multline*}

There are also two types of fixed loci for $\oQ_{0, 2}(\bP^4, 1)$.
The first type of fixed locus is where except for an order one base
point the entire source curve is contracted to a fixed point $i$.
Such a fixed locus is isomorphic to a point.

The contribution of the virtual class for this fixed locus is given by
\begin{equation*}
  \frac 1{\prod_{j \neq i} (\lambda_i - \lambda_j)^2},
\end{equation*}
and the contribution of the twisting is
\begin{equation*}
  (t + 5\lambda_i)^6.
\end{equation*}
On this fixed locus, the descendent classes vanish, and we have
$\ev_1^*(H) = \ev_2^*(H) = \lambda_i$.
Thus, the insertion gives a factor of
\begin{equation*}
  \frac{\frac{25}9 \lambda_i^6 + \frac{95}{36} \lambda_i^7 t^{-1} + \frac{4705}{576} \lambda_i^8 t^{-2}}{(t - 5\lambda_i)^4}
\end{equation*}
In total, the contribution of the fixed loci of first type is
\begin{equation*}
  \frac{1967}{192}.
\end{equation*}

A fixed locus of second type is where the map is a degree one cover of
a fixed line of $\bP^4$ such that marking 1 is mapped to fixed point
$i$ and marking 2 is mapped to fixed point $j$.
Such a locus is again just a point.
The contribution of the virtual class is then given by
\begin{equation*}
  \frac 1{\prod_{k \neq i} (\lambda_i - \lambda_k) \prod_{k \neq j} (\lambda_j - \lambda_k)},
\end{equation*}
and the contribution of the twisting is
\begin{equation*}
  (t + 5\lambda_i) (t + 4\lambda_i + \lambda_j) \dotsb (t + 5\lambda_j).
\end{equation*}
We have $\ev_1^*(H) = \lambda_i, \ev_2^*(H) = \lambda_j$, the
descendent class at marking one is $\lambda_j - \lambda_i$, while the
other descendent class is $\lambda_i - \lambda_j$.
So, the insertion gives a factor of
\begin{equation*}
  \frac{\frac{25}9 \lambda_i^3 \lambda_j^3 + \frac{95}{72} \lambda_i^3 \lambda_j^4 t^{-1} + \frac{95}{72} \lambda_i^4 \lambda_j^3 t^{-1} + \frac{4705}{576} \lambda_i^4 \lambda_j^4 t^{-2}}{(t - 5\lambda_i) (t - 5\lambda_j) (t - 4\lambda_i - \lambda_j) (t - \lambda_i - 4\lambda_j)}
\end{equation*}
Thus, the total contribution is
\begin{equation*}
  \frac{3001}{144}.
\end{equation*}

\subsubsection{Final result}

Collecting all the contributions gives
\begin{equation*}
  -\frac{4285}{1152} + \frac 12 \frac{2965}{288} + \frac{975}{64} - \frac{3425}{288} + \frac 12 \left(\frac{1967}{192} + \frac{3001}{144}\right) = \frac{325}{16},
\end{equation*}
which is the expected coefficient of $q^1$ in $F_2^{SQ}(q)$.

\section{Proof of the Main Theorem} \label{sec:twist}

In this section, we provide a list of closed formulae for the
contribution of each graph in Section~\ref{sec:geo:g2}.
Based on these closed formulae, we prove the main theorem.
One subtlety are the extra generators appearing in the twisted theory.
They mysteriously cancel each other when we sum up the
contributions. We will come back to these cancellations in higher
genus in \cite{GJR}.
The proof of these formulae will be presented in
Section~\ref{proofofproposition}.

\subsection{Genus zero mirror theorem for the twisted theory of $\mathbb P^4$}
Let $I(\t,q,z)$ be the $I$-function of the twisted invariants, that is explicitly
\begin{equation}\label{Itwist}
I(\t,q,z) =  z \, \sum_{d\geq 0 } q^d  \frac{\prod_{j=1}^{5d}  (5H + jz -\t)}{ \prod_{k=1}^d (H+k z)^5 },
\end{equation}
which satisfies the following Picard--Fuchs equation
\begin{equation} \label{PFequation}
\Big( D_H^5 -  q \prod_{k=1}^5 (5D_H+kz-\t) \Big)  I(\t,q,z) = 0,
\end{equation}
where $D_H := D + H := z q\frac{d}{dq} + H$.
The $I$-function has the following form when expanded in $z^{-1}$:
$$
I (\t,q,z)  = z I_0(q) \mathbf 1 + I_1(\t,q)  +z^{-1} I_2(\t,q) + z^{-2} I_3(\t,q)+\cdots
$$

The genus zero mirror theorem \cite{Gi96} relates $I(t, q, z)$ to the
$J$-function, defined by
\begin{equation*}
  J(\bt,z):= -\mathbf 1 z+ \bt(-z) + \sum_i \varphi^i \left<\left< \frac{\varphi_i}{z-\psi}\right>\right>^t_{0,1}(\bt (\psi)),
\end{equation*}
where we define the double bracket for the twisted Gromov--Witten
invariants by
\begin{align*}
  &\left<\left< \gamma_1(\psi),\cdots,\gamma_m(\psi) \right>\right>_{g,m}^t(\bt(\psi)) \\
  =\,& \sum_{n} \frac{1}{n!}\left<\gamma_1(\psi),\cdots,\gamma_m(\psi) ,\bt(\psi),\cdots, \bt(\psi) \right>^t_{g,m+n} \\
  =\,& \sum_{n,d} \frac{q^d}{n!} \int_{[\M_{g,m+n}(\bP^4,d)]^\vir}  e(R^* \pi_*  f^* \cO(5) \otimes [1]) \prod_{j=1}^m \gamma_j(\psi_j)  \prod_{k=m+1}^{m+n} \bt(\psi_k) ,
\end{align*}
where we recall that $[1]$ is a line bundle with first Chern class
$t$, and where $\{\varphi_i\}$ is any basis of $H^*(\bP^4)$ with dual
basis $\{\varphi^i\}$ under the inner product
$$
(a,b)^{\t} : = \int_{\bP^4} a\cup b\cup (5H-\t)  .
$$

To state the precise relationship, we write
$$
I_1(\t,q) = I_{1}(q) H + I_{1;a}(q) \t,
$$
and define the mirror map by
$$
\tau(q)  = \frac{I_1(\t, q)}{I_0(q)} = H   \frac{I_1(q)}{I_0(q)} + \t \frac{I_{1;a}(q)}{I_0(q)}.
$$
Then Givental's mirror theorem states that the $J$-function of the
twisted invariants can be computed from the $I$-function by
$$
J(\tau(q),z)   =\frac{I(\t,q,z)}{I_0(q)}.
$$

\subsection{Quantum product and ``extra'' generators}

We consider the quantum product in the twisted theory, which is defined for any point $\bt \in H_{\mathbb C^*} ^*(\bP^4)$ by
$$
a *_{\bt} b :=  \sum_{i} \varphi^i  \left<\left< a,b, \varphi_i \right>\right>^{\t}_{0,3}(\bt)
= \sum_{i}  \varphi^i \left< a,b, \varphi_i  \right>^{\t}_{0,3}   .
$$
In the basis $\{H^k\}$, the quantum product $\dot\tau *_\tau$,
in which
\begin{equation*}
  \dot\tau  = H + q \frac d{dq} \tau,
\end{equation*}
can be identified with a $5 \times 5$-matrix. It is not hard to see
that it has the following form
\begin{equation*}
  \dot\tau *_\tau = A: = \begin{pmatrix}
    I_{1,1;a}\t  & I_{1,1} \\
    I_{2,2;b}\t^2 & I_{2,2;a}\t  & I_{2,2} \\
    I_{3,3;c}\t^3 & I_{3,3;b}\t^2 & I_{3,3;a}\t & I_{3,3} \\
    I_{4,4;d}\t^4 & I_{4,4;c}\t^3 & I_{4,4;b}\t^2 & I_{4,4;a}\t  & I_{4,4} \\
    I_{5,5;e}\t^5 & I_{5,5;d}\t^4 & I_{5,5;c}\t^3 & I_{5,5;b}\t^2 &
    I_{5,5;a}\t
   \end{pmatrix}^t,
\end{equation*}
where the $I_{i, i}$ and $I_{i, i; *}$ are certain power series in
$q$.
Recall that by the basic theory of Frobenius manifolds, the $S$-matrix
$S(\bt,z)$, which is defined by
$$
S(\bt,z) \varphi^i = \varphi^i + \sum_j \varphi^j \,\left<\left<  \varphi_j, \frac{\varphi_i}{z-\psi} \right>\right>^t (\bt)  ,
$$
is a solution of the quantum differential equation
\begin{equation*}
  d S(\bt,z)  = d\bt *_{\bt}  S(\bt,z) ,
\end{equation*}
where the differential $d$ acts on the coordinates
$\bt \in H_{\mathbb C^*} ^*(\bP^4)$ (but not on the Novikov variable
$q$).
By using the divisor equation and the matrix introduced above, we can
write down the explicit form of the quantum differential equation at
$\bt = \tau(q)$:
\begin{equation*}
  D_H S(\tau(q),z) = A(t,q) \cdot S(\tau(q),z)
\end{equation*}
Since the $S$-matrix can be obtained from the derivatives of the
$I$-function by Birkhoff factorization (see e.g. \cite{CoGi07} and see
also Proposition~\ref{birkhoff} below), we see that all the entries in
this matrix $A$ can be written explicitly in terms of the derivatives
of $I_k$.
In particular, we have
\begin{equation*}
  I_{1,1} = 1+ q\frac{d}{dq} \left(\frac{I_{1}}{I_0} \right),\quad
  I_{1,1;a} = q\frac{d}{dq} \left(\frac{I_{1,a}}{I_0} \right) .
\end{equation*}
\begin{remark}
  One can check that our $I_{p,p}$ coincide with the $I_p$ in Theorem
  1 of Zagier--Zinger's paper \cite{ZZ08}.
\end{remark}

Recall that we have introduced the following degree $k$ ``basic'' generators
\begin{equation*}
  \X_k := \frac{d^k}{du^k} \left(\log \frac{I_0} L\right),\quad
  \Y_k:= \frac{d^k}{du^k} \left(\log \frac{I_0  I_{1,1}}{ L^2}\right),\quad
  \Z_k:=  \frac{d^k}{du^k} \left( \log (q^{\frac{1}{5}}L)\right),
\end{equation*}
By using the entries in the quantum product matrix, we define the following four ``extra'' generators
\begin{align*}
\QQ = \frac{1}{L}\Big(I_{1,1;a}-\frac{1}{5}\Big),\quad
\PP = \frac{I_{1,1}I_{2,2;b}}{L^2}+  \QQ^2+ \frac{ L^4}{2} \QQ, %\quad D \PP:= \frac{d}{du} \PP,\quad D \QQ:= \frac{d}{du} \QQ .
\\
\tilde \QQ:=  { \frac{d}{du} \QQ}+(\X-\Y)\QQ  ,\quad \tilde \PP:=   {  \frac{d}{du} \PP}+(\X+\Y)\PP  .
\end{align*}
Their degrees are defined by
$$
\deg \QQ := 1,\quad \deg \PP:=2 ,\quad \deg \frac{d}{du}  :=1,
$$
so that we have
$$
\deg \tilde \QQ = 2,\quad \deg \tilde \PP= 3.
$$
\begin{remark}
We will see that, in the genus $2$ case, only linear terms of the following two extra generators are involved: $\tilde \PP$ and $\tilde \QQ$  .
\end{remark}
\begin{remark}
  \label{Zkpolynomial}
  For any $k$, the generator $\Z_k$ can be written as a
  (non-homogeneous) polynomial of $L$.
  This is because $\Z_1$ is a polynomial of $L$ and
  \begin{equation*}
    \frac{d}{du} L =  \frac{1}{5}(L^5-1) .
  \end{equation*}
  For example, the first several of them are
  \begin{align*}
    \Z_1 =\, & \frac{1}{L}(q\frac{d}{dq}(\log L) +\frac{1}{5})  =\frac{L^4}{5 } ,\\
    \Z_2 = \,& \frac{4}{25 } L^3 (L^5-1)  ,\\
    \Z_3 =  \,&  \frac{4}{125}L^2 ( L^5-1) (8 L^5-3).
  \end{align*}
\end{remark}

\begin{remark} \label{numer}
  Some leading terms of the basic generators are given by
  \begin{align*}
    \X = & -505 q-1425100 q^2-4155623250 q^3 + O(q^4)\\
    \Y = & -360 q-1190450 q^2-3759611500 q^3 + O(q^4)\\
    \X_2 = & -505 q-2534575 q^2-10290963500 q^3 + O(q^4)\\
    \X_3 = & -505 q-4753525 q^2-27310140500 q^3 + O(q^4)\\
    L = & \,\, 1+625 q+1171875 q^2+2685546875 q^3 + O(q^4),
  \end{align*}
  and some leading terms of the extra generators are given by
  \begin{align*}
    \QQ = & - \frac{1}{5}-149 q-271030 q^2-591997100 q^3 + O(q^4)\\
    \PP = &  -\frac{3}{50}-\frac{399}{10} q-12732 q^2+131454705 q^3 + O(q^4)\\
    \tilde \QQ= & -120 q-473525 q^2-1622526750 q^3 + O(q^4)\\
    \tilde \PP= & \,\, 12 q+\frac{331965}{2} q^2+984651825 q^3 + O(q^4) .
  \end{align*}
  In particular, from this data one can see the degree zero genus two
  Gromov--Witten invariant should be equal to the coefficient of the
  $\Z^3$ in the homogenous polynomial $5^3 \frac{L^2}{I_0}F_2 $ (see
  Conjecture~\ref{mainconj}), which is $\frac{-5}{144}$.
\end{remark}

\subsection{A list of closed formulae for the contribution of localization graphs}

We now give a list of closed formulae for the contribution of each of
the localization graphs in Section~\ref{sec:geo:g2}.
We also rewrite each contribution in terms of Gromov--Witten double
brackets using the wall-crossing formula\footnote{While no proof
  exactly applies to our situation, the proofs \cite{CFK17},
  \cite{CFK} and \cite{CJR2} can be easily adapted.}:
\begin{equation*}
  \langle\rangle_{g, n}^{t, SQ}
  = I_0^{2g - 2} \langle\langle \rangle\rangle_{g, n}^{\t} (\tau(q))
\end{equation*}

The first proposition concerns the purely twisted contribution
$\Cont_{\Gamma^2} = L^{-2} \Cont'_{\Gamma^2}$, where
\begin{equation*}
  \Cont'_{\Gamma^2} := (L/I_0)^2 \cdot \left<\left<   \right>\right>_{2,0}^{\t} (\tau(q)).
\end{equation*}
\begin{proposition}
  \label{g2cont}
  The contribution of $\Gamma^{2}$ is a degree $3$ homogeneous
  polynomial in the basic and extra generators, to be precise
  {\small
  \begin{align*}
\Cont'_{\Gamma^2} = &  \,
-{\frac {5}{1152}}\tilde \PP -\big({\frac {31\,\X}{576}}+\frac{1}{48}\,\Y +{\frac {205\,\Z_{{1}}}{2304}} \big) \cdot \tilde  \QQ  \\
&     -{\frac {19\,\X_{{3}}+67\,{\X}^{3}}{1152}}-\frac{ \Y
(\X_{{2}}+5\,{\Y} \X)}{48}-{\frac {25\,\X_{{2}}(\X+\Z_1)}{288}} -
{\frac {101\,\Y{\X}^{2}}{1152}}-\frac{{\Y}^{3}}{24}-{\frac {107\,\Z_{{1}}{\X}^{2}}{384}} \\
&-{\frac {19\,\Z_{{1}}\Y\X}{64}}-\frac{3\Z_
{{1}}{\Y}^{2}}{16}-{\frac {65\,\Z_{{2
}}\X}{1536}}-{\frac {715\,{\Z_{{1}}}^{2}\X}{1152}} -{\frac {45\,{\Z_{{1}}}^{2}\Y}{128}}-{\frac {829\,\Z_{{1}}\Z_{
{2}}}{1728}}+{\frac {349\,\Z_{{3}}}{13824}}
  \end{align*}}
\end{proposition}
The leading terms of the genus two graph contribution are
\begin{align*}
  \Cont_{\Gamma^{2}} = L^{-2}\Cont'_{\Gamma^2}  ={\frac {1645\,q}{1152}}+{\frac {1842665\,{q}^{2}}{576}}+{\frac {
2419134175\,{q}^{3}}{288}}
 +O(q^4).
\end{align*}
Note that the degree $1$ term coincides with the one in the localization computation in Section~\ref{sec:geo:ex}:
$$
\frac {1645}{1152} = -\frac{4285}{1152}+\frac{1}{2}\cdot\frac{ 2965}{288} .
$$

We next consider the contribution from the graph $\Gamma^1$, which is
a genus-one twisted theory with a special insertion.
It can be rewritten as $\Cont_{\Gamma^1} = L^{-2} \Cont'_{\Gamma^1}$,
where
\begin{equation*}
  \Cont'_{\Gamma^{1}} := \frac{L^2}{I_0} \left<\left<  \frac{-\frac{5}{3}H^3+\frac{5}{24}H^4 \t^{-1}}{(\t-5H)(\t-5H-\psi)} \right>\right>_{1,1}^{\t} (\tau(q)).
\end{equation*}
\begin{proposition}
  \label{g1cont}
  The contribution of $\Gamma^{1}$ is a degree $3$ homogeneous
  polynomial in the basic and extra generators, to be precise
{\small
\begin{align*}
-\Cont'_{\Gamma^1} = & \,\,\,\,
{\frac {473 }{576}} \tilde \PP +\big(\frac{1}{48}\Y-{\frac {25}{96}\X}+{\frac {2093}{1152}\,\Z_{{1}
}}\big) \cdot \tilde \QQ   \\
&
+{\frac {25\,(\X_{{3}}+\X^3)}{72}}+{
\frac {41\,\Y(\X_{{2}}-{\Y}\X-{\Y}\Z_1+\Z_2)}{48}}+{\frac {1871\,\X\X_{{2}}- {1271\,\Y{\X}^{2}}-{  {1471\,\Z_{{1}}{
\X}^{2}} }  }{576}}\\& +{\frac {1025\,\Z_{{1}}\X_{{2}}}{288}}  -{\frac {451\,\Z_{{1}}\Y\X}{72}}+
{\frac {1267\,\Z_{{2}}\X}{1152}}+{\frac {1039\,{\Z_{{1}}}^{2}\X}{576}}  -{\frac {779
\,{\Z_{{1}}}^{2}\Y}{192}}+{\frac {4945\,\Z_{{1}}\Z_{{2}}}{864}}-{\frac {
155\,\Z_{{3}}}{3456}}
\end{align*}}
\end{proposition}
The leading terms of the genus one graph contribution are
\begin{align*}
 - \Cont_{\Gamma^{1}} = -L^{-2}\Cont'_{\Gamma^1}  ={\frac {1925\,q}{576}}-{\frac {2344025\,{q}^{2}}{288}}-{\frac {
4831529575\,{q}^{3}}{144}}
 +O(q^4).
\end{align*}
Note that the degree $1$ term coincides with the one in the localization computation in Section~\ref{sec:geo:ex}:
$$
\frac{1925}{576} = \frac{975}{64}-\frac{3425}{288}
$$

The remaining two (non-trivial) graphs involve a genus-zero
two-pointed twisted theory.
We rewrite them as $\Cont_{\Gamma^0_a} = L^{-2} \Cont'_{\Gamma^0_a}$ and $\Cont_{\Gamma^0_b} = L^{-2} \Cont'_{\Gamma^0_b}$,
where
\begin{align*}
  \Cont'_{\Gamma_{a}^{0}}:= &  {L^2}  \left<\left<  \frac{ \frac{5}{3}H^3 \otimes H^4 \t^{-1}+\frac{5}{3} H^4 \otimes H^3 \t^{-1} +\frac{65}{8} H^4 \otimes H^4 \t^{-2} }{(\t-5H)(\t-5H-\psi_1) \, (\t-5H)(\t-5H-\psi_2)} \right>\right>_{0,2}^{\t}(\tau(q)),
  \\
  \Cont'_{\Gamma_{b}^{0}}:= &  {L^2} \left<\left<  \frac{-\frac{5}{3}H^3+\frac{5}{24}H^4 \t^{-1}}{(\t-5H)(\t-5H-\psi_1)},
\frac{-\frac{5}{3}H^3+\frac{5}{24}H^4 \t^{-1}}{(\t-5H)(\t-5H-\psi_2)} \right>\right>_{0,2}^{\t}(\tau(q)).
\end{align*}

\begin{proposition}
  \label{g0contp}
  The contributions of $\Gamma_{a}^{0}$ and $\Gamma_{b}^{0}$ are both
  degree $3$ homogeneous polynomials in the basic and extra generators, to
  be precise
  {\small
  \begin{align*}
    \Cont'_{\Gamma_{a}^{0}} =&-{\frac {13 }{8}}  \tilde \PP -\big(\frac{1}{12}\X +{\frac {199 }{48}\Z_{{1}}}\big)\cdot  \tilde \QQ\\
                            &
                              +{\frac {41}{24}}({\X}^{3} +{\X}^{2}\Y +{\Z_{{1}}}^{2}\Y) +{\frac {77\,\Z_{{1}}{\X}^{2}}{24}}+{\frac {
                              41\,\Z_{{1}}\Y\X}{12}}\\
                            & +{\frac {677\,\Z_{{2}}\X}{96}}-{\frac {277\,{\Z_{{1}}}
                              ^{2}\X}{24}} -{\frac {599\,\Z_{{1}}\Z_{{2
                              }}}{72}}+{\frac {805\,\Z_{{3}}}{288}},\\
    \Cont'_{\Gamma_{b}^{0}} = &-{\frac {5}{576}}  \tilde \PP  +
                               \big({\frac {205\X }{288}}+{\frac {265\Z_{{1}}}{384}} \big)\cdot  \tilde \QQ\\
                            &+{\frac {7595}{576}\,(\X_{{3}}+{\X}^{3})}
                              +{\frac {15595\,\X\X_{{2}}}{288}} -{\frac {8405\,\Y }{576}(\X^2+2\X\Z_{{1}}+\Z_{{1}}^{2})}
    \\&+\big({\frac {250\X_{{2}}}{9}}+{\frac {215
        {\X}^{2}}{576}}\big)\Z_{{1}} +{\frac {117215
        \,\Z_{{2}}\X}{2304}}-{\frac {2405\,{\Z_{{1}}}^{2}\X}{192}} +{\frac {1585\,\Z_{{1}}\Z_{{2}}}{64}}+{\frac {30325
        \,\Z_{{3}}}{2304}} .
  \end{align*}}
\end{proposition}

\begin{corollary}  \label{g0cont}
  The summation $\Cont'_{\Gamma^0}$ of the two contributions is
{\small
  \begin{align*}
    \Cont'_{\Gamma^0}  = &-{\frac {941}{576}}\,{ \tilde\PP}  +\big({\frac {181}{288}}\X-{\frac {1327}{384
}\Z_1}\big) \cdot  \tilde \QQ   \\ &
+{\frac {8579\,(\X_{{3}}+{\X}^{3})
}{576}}+{\frac {16579\,\X\X_{{2}}}{288}} -{\frac {7421\,\Y{\X} (\X+2\Z_1)}{576}}+{\frac {250\,\Z_{{1}}\X_{{2}}}{9}
}+{\frac {
2063\,\Z_{{1}}{\X}^{2}}{576}}  \\&
-{\frac {
4621\,{\Z_{{1}}}^{2}\X}{192}}+{\frac {133463\,\Z_{{2}}\X}{2304}}-{\frac {
7421\,{\Z_{{1}}}^{2}\Y}{576}}
+{\frac {4085\,\Z_{{3}}}{256}}+{\frac {9473
\,\Z_{{1}}\Z_{{2}}}{576}}.
\end{align*}}
\end{corollary}
The leading terms of the genus zero graph contribution are
\begin{align*}
   \Cont_{\Gamma_{a}^{0}}+  \Cont_{\Gamma_{b}^{0}}= L^{-2}\Cont'_{\Gamma^0}  = {\frac {17905}{576}}\,q+{\frac {11650385}{288}}\,{q}^{2}+{\frac {
13428251725}{144}}\,{q}^{3} +O(q^4).
\end{align*}
Note that the degree $1$ term coincides with the one in the localization computation in Section~\ref{sec:geo:ex}:
$$
\frac{17905}{576} = \frac{1967}{192} + \frac{3001}{144}
$$

\subsection{Proof of the Main Theorem}

Our main result now follows from the propositions in the last subsection.

\begin{theorem}
  The genus two Gromov--Witten free energy of a quintic Calabi--Yau
  threefold is given by
\begin{align*}
 F^{GW}_2(Q) = \, &
\frac{I^2_0 }{L ^2} \cdot \Big( {\frac {70\,\X_{{3}}}{9}}+{\frac {575\,\X\X_{{2}}}{18}}+\frac{5 \Y\X_{{2}}}{6}+{\frac {557\,{\X}^{3}}{72}}-{\frac {629\,\Y{\X
}^{2}}{72}}-{\frac {23\,{\Y}^{2}\X}{24}}-\frac{{\Y}^{3}}{24}\\&+{
\frac {625\,\Z\X_{{2}}}{36}}-{
\frac {175\,\Z\Y\X}{9}}+{\frac {1441\,\Z_{{2
}}\X}{48}}-{\frac {25\,\Z({\X}^{2}+{\Y}^{2})}{24}}-{\frac {3125\,{\Z}^{2}(\X+\Y)}{288}}\\
&  +{\frac {41\,\Z_{{2}}\Y}{48}}-{\frac {625\,{\Z}^{3}}{144}}+{
\frac {2233\,\Z\Z_{{2}}}{128}}+{\frac {547\,\Z_{{3}}}{72}} \Big) .
 \end{align*}
\end{theorem}
\begin{proof}

  By dimension considerations, the $\Cont'_{\Gamma^i}$ have no
  constant term in $q$.
  Therefore, the value of $c_{2, 0} = - N_{2, 0}$ can be read off from
  (see \cite{Pa02})
  \begin{equation*}
    N_{2, 0}
    = \frac 12 \int_{Q_5} (c_3(Q_5) - c_1(Q_5) c_2(Q_5)) \cdot \int_{\M_{2, 0}} \lambda_1^3
    = \frac 12 \cdot (-200) \cdot \frac{|B_4|}4 \cdot \frac{|B_2|}2 \cdot \frac 1{2!}
    = -\frac 5{144}.
  \end{equation*}

  The rest is just a direct consequence of the following formula
  \begin{align*}
    L^2 \cdot F^{SQ}_2(q) =  -c_{2, 0} L^2 + \frac{1}{2} \Cont'_{\Gamma^0} -\Cont'_{\Gamma^1} +\Cont'_{\Gamma^2} ,
  \end{align*}
  the wall-crossing formula
  \begin{equation*}
    F^{GW}_2(Q)  =  I_0(q)^{ 2} \cdot F^{SQ}_2(q) ,
  \end{equation*}
  and the formulae for $\Cont_i$ in Proposition \ref{g2cont},
  Proposition \ref{g1cont} and Corollary \ref{g0cont}.
\end{proof}

\subsection{Equivalence between our Main Theorem and the physicists' conjecture} \label{sec:equivalence}

The closed formula for the genus two Gromov--Witten potential was
first proposed by BCOV \cite{BCOV} and further clarified by
Yamaguchi--Yau \cite{YY04}.
Later, Klemm--Huang--Quackenbush  \cite{HKQ09} extended the result to genus 51.
We will follow the notation of \cite{HKQ09}.

In \cite{YY04, HKQ09}, the authors introduce the following basic
generators
\begin{equation*}
  A_p:=\frac{(- q \frac{d}{dq})^p \big( q I_{1,1}(q)\big)}{ q I_{1,1}(q)},\quad
  B_p:= \frac{(- q \frac{d}{dq})^p I_0(q)}{I_0(q)} ,\quad
  X:= \frac{-5^5 q}{1-5^5 q}
\end{equation*}
and the change of variables
$$
B_1 = u ,\quad  A_1 = v_1-1-2u\quad B_2 = v_2 + u v_1,\quad
B_3 = v_3 - u v_2 + u v_1X-\frac{2}{5}u X .
$$
The following is the original physical conjecture.
\begin{conjecture}\label{YY}
Let $F_g$ be the genus $g$ Gromov--Witten potential and
$$
P_g (q):=  \big( \frac{-5}{I_0^2(1-5^5q)}\big)^{g-1} F_g(Q) .
$$
When $g=2$, we have the following explicit formula
\begin{align*}
P_2 =& \,\, {\frac {25}{144}}-{\frac {625\,{v_1}}{288}}+{\frac {25\,{{v_1}}^
{2}}{24}}-{\frac {5\,{{v_1}}^{3}}{24}}-{\frac {625\,{v_2}}{36}}+
{\frac {25\,{v_1}\,{v_2}}{6}}+{\frac {350\,{v_3}}{9}}-{\frac
{5759\,X}{3600}}\\
 &\,\, -{\frac {167\,{v_1}\,X}{720}}+\frac{{{v_1}}^{2}X}{6}-
{\frac {475\,{v_2}\,X}{12}}+{\frac {41\,{X}^{2}}{3600}}-{\frac {13
\,{v_1}\,{X}^{2}}{288}}+{\frac {{X}^{3}}{240}}.
\end{align*}
\end{conjecture}

\begin{proposition}
Both Conjecture~\ref{YY} and Conjecture~\ref{mainconj} are equivalent to
\begin{equation}
  \label{F2formula}
  \begin{aligned}
    F_2^{GW}(Q) = \frac{I_0^2}{L^2} \cdot &\,\Big({\frac {70\,{\X_3}}{9}}+{\frac {575\,\X{\X_2}}{18}}+{\frac {557\,{\X}^{3}}{72}}+\frac{5\,\Y{\X_2}}{6}-{\frac {629\,\Y{\X}^{2}}{72}}-{\frac {23\,{\Y}^{2}
\X}{24}}-\frac{{\Y}^{3}}{24}\\ &\,+{\frac {125\,{
\X_2}\,{L}^{4}}{36}}-{
\frac {5\,{\X}^{2}{L}^{4}}{24}}-{\frac {35\,\Y\X {L}^{4}}{9}}-{
\frac {5\,{\Y}^{2}{L}^{4}}{24}}-{\frac {1441\,\X{L}^{3}}{300}}-{\frac {41\,\Y{L}^{3}}{300
}}\\ &\,
+{\frac {
31459\,\X {L}^{8}}{7200}}-{\frac {2141\,\Y{L}^{8}}{7200}}+{\frac {29621\,{L}^{12}}{12000}}-{\frac {116369\,{L}^{7}}{36000}}+{\frac {547\,{L}^{2}}{750}}\Big).
  \end{aligned}
\end{equation}
\end{proposition}
\begin{proof}
First notice that
\begin{align*}
&u = B_1, \quad  v_1 = A_1+1+2 B_1,\quad  v_2 = -A_1 B_1-2 B_1^2-B_1+B_2, \\
&v_3 = -A_1 B_1^2-B_1 A_1 X-2 B_1^3-2 B_1^2 X-B_1^2+B_1 B_2-\frac{3}{5} B_1 X+B_3.
\end{align*}
Therefore, Conjecture~\ref{YY} is equivalent to
{\footnotesize
\begin{align*}
P_2 = &\, {\frac {385\,A_1B_1}{36}}-{\frac {1045\,A_1{B}_1^{2}}{18}}+{\frac {5923\,B_1X}{
360}}+{\frac {37\,{B}_1^{2}X}{18}}+{\frac {425\,B_1{B_2}}{9}}+{\frac {
565\,B_1}{48}}-{\frac {205\,A_1}{288}}
-{\frac {13\,{X}^{2}B}{144}} \\
&\,-{\frac {65\,{A_1}^{2}B_1}{12}}+{\frac {25\,A_1{B_2}}{6}}+\frac{{A_1}^{2}X}{6}
 -{\frac {475\,{B_2}\,X}{12}}-{\frac {5\,{A_1}^{3}}{24}}+{
\frac {{X}^{3}}{240}}-{\frac {27\,{X}^{2}}{800}}+{\frac {5\,{A_1}^{2}}{
12}}+{\frac {73\,XA_1}{720}} \\
&\,-{\frac {13\,{X}^{2}A_1}{288}}+{\frac {49\,B_1XA_1
}{36}} -{\frac {865\,{B}_1^{3}}{9}}-{\frac {115\,{B}_1^{2}}{6}}+{\frac {350
\,{B_3}}{9}}-{\frac {333\,X}{200}}-{\frac {475\,{B_2}}{36}}-{\frac {335}{288}} .
\end{align*}}%
Next we define
$$
\tilde A_p:= \Big(- q \frac{d}{dq}\Big)^p \log (q^{\frac{1}{5}}I_0),  \quad \tilde B_p:= \Big(- q \frac{d}{dq}\Big)^p  \log (q^{\frac{1}{5}}I_{1,1}),
$$
so that
\begin{align*}
&A_1 = -\frac{4}{5}+\tilde A_1,\quad  B_1 = \frac{1}{5}+\tilde B_1,\quad B_2 =  \frac{1}{25}+\frac{2}{5} \tilde B_1+(\tilde B_2+\tilde B_1^2), \\
&\quad  B_3 = \frac{1}{125}+\frac{3}{25}\tilde B_1+\frac{3}{5}(\tilde B_2+\tilde B_1^2)+(\tilde B^3+3\tilde  B_1\tilde  B_2+\tilde B_3) . %,\quad B_4=  \frac{1}{625}+\frac{4}{125}\tilde B_1+\frac{6}{25}\tilde B_2+\frac{4}{5}\tilde B_3+\tilde B_4.
\end{align*}
On the other hand, by definition of $\X_p$, $\Y_p$ and $L$, we have
\begin{align*}
&\tilde A = {-L}(\Y+\Z-\X),\quad \tilde B_1 ={-L}(\X+\Z),\quad \tilde B_2 = {L^2} (\X_2+\Z_2)-\frac{L}{5} X (\X+\Z),  \\
&\quad \tilde B_3 =  - L^3 (\X_3+\Z_3)+\frac{3}{5}L^2 X (\X_2+\Z_2)-\frac{L}{25}(6 X^2-5 X) (\X+\Z) .%\\(L^4 (\X_1+\Z_1)4+4 L^3 X (\X_3+\Z_3)+L^2 (27 X^2-20 X) (\X_2+\Z_2)+L (66 X^3-90 X^2+25 X) (\X_1+\Z_1)) (1/625)
\end{align*}
Also $X =1-L^5$ and the $\Z_k$ are all polynomials of $L$ (see Remark \ref{Zkpolynomial}).
Finally, a few direct computations show that, after the above change of variables,  both conjectures are equivalent to equation
\eqref{F2formula}.
\end{proof}

\section{Structures of the twisted invariants}
\label{sec:structure}

The twisted theory of $\bP^4$ is semisimple, and can be computed by
the Givental--Teleman formula using $R$-matrices.
In this section, we write down the basic data and relations for the
twisted invariants, and then we derive closed formulae for the entries of the $R$-matrix up to
$z^3$, which is all we need for the computation in genus two.

\subsection{$S$-matrix}

Recall that the $S$-matrix is a fundamental solution of the quantum differential equation
\begin{equation}\label{QDE}
  D_H S(\tau(q),z) = A(t,q) \cdot S(\tau(q),z),
\end{equation}
where we recall that $D:= z \,q\frac{d}{dq}$ and $D_H := D + H$.
Moreover,  the $S$-matrix can be obtained from the derivatives of the
$I$-function by Birkhoff factorization. We start from $I$-function.
In the rest of this and next sections, we will fix the base point
$$\bt = \tau(q)=H  \frac{I_1(q)}{I_0(q)} + \t \frac{I_{1;a}(q)}{I_0(q)}.$$
As a convention, we will omit the base point $\bt$ in the double
bracket $\left<\left< - \right>\right>^t_{g,n}(\bt)$ and the
$S$-matrix $S(\tau(q), z)$ when $\bt =\tau$.

The following proposition is an application of Birkhoff factorization.
\begin{proposition}
  \label{birkhoff}
  We have the following formulae for the $(\cdot, \cdot)^t$-adjoint of the
  $S$-matrix
\begin{align*}
S^*(z)(\mathbf 1) =\,&   \frac{I(z)}{z I_0} \\
S^*(z)(H) =\,&  \left( \frac{D_H-I_{1,1;a}\t}{I_{1,1}}\right)  \frac{I(z)}{z I_0}\\
S^*(z)(H^2) =\,& \det\begin{pmatrix}
\frac{D_H-I_{2,2;a}\t}{I_{2,2}} & -\frac{I_{2,2;b}}{I_{2,2}}\t^2  \\
-1& \frac{D_H-I_{1,1;a}\t}{I_{1,1}} \end{pmatrix} \frac{I(z)}{z I_0 }\\
S^*(z)(H^3) =\,& \det\begin{pmatrix}
\frac{D_H-I_{3,3;a}\t}{I_{3,3}} & -\frac{I_{3,3;b}}{I_{3,3}}\t^2& -\frac{I_{3,3;c}}{I_{3,3}}\t^3 \\
-1 & \frac{D_H-I_{2,2;a}\t}{I_{2,2}} & -\frac{I_{2,2;b}}{I_{2,2}}\t^2  \\
&-1& \frac{D_H-I_{1,1;a}\t}{I_{1,1}} \end{pmatrix}\frac{I(z)}{z I_0} \\
S^*(z)(H^4) =\,& \det\begin{pmatrix}
\frac{D_H-I_{4,4;a}\t}{I_{4,4}} & -\frac{I_{4,4;b}}{I_{4,4}}\t^2& -\frac{I_{4,4;c}}{I_{4,4}}\t^3& -\frac{I_{4,4;d}}{I_{4,4}}\t^4 \\
-1&\frac{D_H-I_{3,3;a}\t}{I_{3,3}} & -\frac{I_{3,3;b}}{I_{3,3}}\t^2& -\frac{I_{3,3;c}}{I_{3,3}}\t^3 \\
&-1 & \frac{D_H-I_{2,2;a}\t}{I_{2,2}} & -\frac{I_{2,2;b}}{I_{2,2}}\t^2  \\
&&-1& \frac{D_H-I_{1,1;a}\t}{I_{1,1}} \end{pmatrix}\frac{I(z)}{z I_0} .
\end{align*}
Note that there is a differential operator $D_H$ in the determinants.
We define the differential operation from top to bottom.
\end{proposition}
\begin{proof}
  Noting that the $S$-matrix is a solution of equation \eqref{QDE},
  and that $S(z)\varphi^i = \varphi^i + O(z^{-1})$, this
  proposition follows from a direct computation: {\footnotesize
\begin{align*}
S^*(z)(\mathbf 1) =\,&   \frac{I(z)}{zI_0} \\
S^*(z)(H) =\,& \frac{1}{I_{1,1}}\Big( {D_H-I_{1,1;a} \t} \Big) S^*(z) \mathbf 1   \\
S^*(z)(H^2) =\,& \frac{1}{I_{2,2}}\Big( {D_H-I_{2,2;a} \t} \Big) S^*( z)H -\frac{I_{2,2;b}}{I_{2,2}}\t^2 S^*( z)\mathbf 1  \\
S^*(z)(H^3) =\,& \frac{1}{{I_{3,3}}}\Big( ({D_H-I_{3,3;a} \t}) S^*( z)H^2 -I_{3,3;b}\t^2  S^*( z)H -I_{3,3;c}\t^3 S^*( z)\mathbf 1 \\
S^*(z)(H^4) =\,& \frac{1}{{I_{4,4}}}\Big( ({D_H-I_{4,4;a} \t}) S^*( z)H^3-I_{4,4;b}\t^2  S^*( z)H^2 -I_{4,4;c}\t^3 S^*( z) H -I_{4,4;d}\t^4 S^*( z)\mathbf 1 \Big)
\end{align*}
}
\end{proof}

Now we can write down all the entries of $S$-matrix.
However, it soon becomes too complicated to write down all the
explicit formulae for further computations.
 We want to establish an equation satisfied by
$$
 S^* (z)\Big({\frac {{\mathbf 1}}{\t}}+5\,{\frac {H}{{\t}^{2}}}+25\,{
\frac {H^2}{{\t}^{3}}}+125\,{\frac {H^3}{{\t}^{4}}
}+625\,{\frac {H^4}{{\t}^{5}}} \Big).
$$
This equation will help us to simplify some computations involving the
$S$-matrix (for example the computation of the modified $V$-matrix in
Section~\ref{sec:g0cont}).
Also by using this equation we can deduce closed formulae for some
special entries of the $R$-matrix and prove some important identities.

\begin{lemma}\label{Hdual4}
Define
$$
\tilde H_4:= {\frac {{\mathbf 1}}{\t}}+5\,{\frac {H}{{\t}^{2}}}+25\,{
\frac {H^2}{{\t}^{3}}}+125\,{\frac {H^3}{{\t}^{4}}
}+625\,{\frac {H^4}{{\t}^{5}}}.
$$
Then, we have
\begin{equation}
  \left( {D_H-\frac{1}{5}\t}\right) S^* (z)(\tilde H_4)+ \frac{1 }{ 5} I(z)=0
\end{equation}
Furthermore,
$$
 S^* (t-5H)(\tilde H_4 ) =\,S^*\left(\frac{t-5H}{2} \right)(\tilde H_4 )   = \,    \tilde H_4.  $$
\end{lemma}
\begin{remark}
  The $\tilde H_4$ in this lemma can be viewed an element of the dual
  basis, which satisfies
$$(\tilde H_4, \mathbf 1)=(\tilde H_4, H)=(\tilde H_4, H^2)=(\tilde H_4, H^3)=0.$$
\end{remark}
\begin{proof}

We define formally
{\small
\begin{align*}
S_5^*
=\,&  \det\begin{pmatrix}
 {D_H-I_{5,5;a}\t}  & - {I_{5,5;b}} \t^2& - {I_{5,5;c}} \t^3 & - {I_{5,5;d}} \t^4 & - {I_{5,5;e}} \t^5\\
-1&\frac{D_H-I_{4,4;a}\t}{I_{4,4}} & -\frac{I_{4,4;b}}{I_{4,4}}\t^2& -\frac{I_{4,4;c}}{I_{4,4}}\t^3& -\frac{I_{4,4;d}}{I_{4,4}}\t^4 \\
&-1&\frac{D_H-I_{3,3;a}\t}{I_{3,3}} & -\frac{I_{3,3;b}}{I_{3,3}}\t^2& -\frac{I_{3,3;c}}{I_{3,3}}\t^3 \\
&&-1 & \frac{D_H-I_{2,2;a}\t}{I_{2,2}} & -\frac{I_{2,2;b}}{I_{2,2}}\t^2  \\
&&&-1& \frac{D_H-I_{1,1;a}\t}{I_{1,1}} \end{pmatrix}\frac{I(z)}{z I_0}
\end{align*}
}

Since $H^5=0$, we have
$$S_5^* = S^*(z)(H^5)=0.$$
On the other hand, by symmetry of the quantum product (see Section~\ref{Ikrelations} for more details):
\begin{align*}
  I_{5,5;a} = &\frac{1}{5} (1-I_{4,4}) \\
  I_{5,5;b} = &\frac{1}{25} (1-I_{3,3})-\frac{1}{5}{I_{4,4;a}}\\
  I_{5,5;c} = &\frac{1}{125} (1-I_{2,2})-\frac{1}{25}{I_{3,3;a}}-\frac{1}{5}{I_{4,4;b}}\\
  I_{5,5;d} = &\frac{1}{625} (1-I_{1,1})-\frac{1}{125}{I_{2,2;a}}-\frac{1}{25}{I_{3,3;b}}-\frac{1}{5}{I_{4,4;c}}\\
  I_{5,5;e} = &\frac{1}{3125} (1-I_{0})-\frac{1}{625}{I_{1,1;a}}-\frac{1}{125}{I_{2,2;b}}-\frac{1}{25}{I_{3,3;c}}-\frac{1}{5}{I_{4,4;d}} .
\end{align*}
Hence by replacing the first row of the matrix with the first row plus the sum of all the $(k+1)$-th row multiplied by $\frac{t^k}{5^k}I_{5-k,5-k}$ for $k=1,2,3,4$, we have
{\small
\begin{align*}
S_5^*
=\,&  \det\begin{pmatrix}
 {D_H-\frac{1}{5}\t}  & \frac{\t}{5}  D_H- \frac{\t^2 }{25}  & \frac{\t^2}{25}  D_H- \frac{\t^3 }{125}  & \frac{ \t^3 }{125} D_H- \frac{\t^4 }{625} &    \frac{\t^4}{625}   D_H-\frac{\t^5}{3125} (1-I_{0}) \\
-1&\frac{D_H-I_{4,4;a}\t}{I_{4,4}} & -\frac{I_{4,4;b}}{I_{4,4}}\t^2& -\frac{I_{4,4;c}}{I_{4,4}}\t^3& -\frac{I_{4,4;d}}{I_{4,4}}\t^4 \\
&-1&\frac{D_H-I_{3,3;a}\t}{I_{3,3}} & -\frac{I_{3,3;b}}{I_{3,3}}\t^2& -\frac{I_{3,3;c}}{I_{3,3}}\t^3 \\
&&-1 & \frac{D_H-I_{2,2;a}\t}{I_{2,2}} & -\frac{I_{2,2;b}}{I_{2,2}}\t^2  \\
&&&-1& \frac{D_H-I_{1,1;a}\t}{I_{1,1}} \end{pmatrix}\frac{I(z)}{z I_0}
\end{align*}
}%
Note that, in general, the determinant will change when we perform a
row transformation for a matrix containing the operator $D_H$.
However in this case, it does not since there are only constant terms
$-1$ under the diagonal.
By Proposition~\ref{birkhoff}, we obtain
$$
0 = S_5^* = \Big( {D_H-\frac{1}{5}\t}\Big) \Big( S_4^*+\frac{\t}{ 5} S_3^*+\frac{\t^2}{ 25} S_2^*+\frac{\t^3}{125} S_1^*+\frac{\t^4}{625} S_0^*\Big)+ \frac{\t^5}{3125} \frac{I(z)}{z}=0
$$
where $S_k:= S(z)^*(H^k)$ for $k=0,1,2,3,4$.
This proves the first statement of the lemma.

In particular, letting $z=\t-5H$, we have $D_H = H+(\t-5H)q\frac{d}{dq}$, and the equation becomes
$$
  ( \t-5H)\left(-\frac{1}{5}+ q\frac{d}{dq}\right)   S^* (\t-5H)(\tilde H_4 ) + \frac{1}{5} =0 .
$$
We can solve this equation using the initial condition
$$
  S^* (z)(H^k ) |_{q=0}= H^k .
$$
It implies the second statement of the lemma.
The third can be proved similarly.
\end{proof}

\subsection{$\Psi$-matrix and $R$-matrix: computations of $\Psi \mathbf 1$ and  $R^* \mathbf 1$}

The twisted theory of $\bP^4$ is semisimple\footnote{Semisimplicity
  will follow from the computations in this section.}, in the sense
that there exist idempotents $e_\alpha$ with respect to the quantum
product (at $\bt = \tau(q)$):
$$
e_\alpha  *_\tau e_\beta  = \delta_{\alpha\beta} e_\alpha
$$
In addition, we recall the definition of the normalized canonical
basis
$$
\bar e_\alpha  :=  \Delta_\alpha^{\frac{1}{2}}  e_\alpha , \qquad
\Delta_\alpha ^{-1}: = (e_\alpha , e_\alpha )^{t}.
$$

By results of Dubrovin and Givental \cite{Du96,Gi98,Gi01a}, there
exists an asymptotic fundamental solution of the quantum differential
equation \eqref{QDE} which has the following form
\begin{equation} \label{funsol}
  \tilde S(\bt,z) =  \Psi^{-1}(\bt) R(\bt, z) e^{U(\bt)/z}  ,
\end{equation}
where $\Psi^{-1}$ is the change of basis from a flat basis to the
normalized canonical basis, $U$ is a diagonal matrix with entries the
canonical coordinates
$$U=\diag(u^0,u^1,\dotsc,u^4)$$
and $R(\bt, z) = 1 + O(z)$ is a matrix of formal power series in $z$.
While there is no direct relation between $\tilde S$ and the
$S$-matrix defined by two point correlators in
Section~\ref{sec:twist}, in the proof of Lemma~\ref{asymptoticofI2} we
will discuss a relation between their fully equivariant
generalizations.

We rewrite the fundamental solution in coordinates as
\begin{equation}\label{asymptoticofS}
  \tilde S_{i\alpha}(z) =   R_{i \bar \alpha}(z) e^{u^\alpha /z}  =\sum_\beta \Psi_{i \bar\beta}  R_{\bar \beta \bar \alpha}(z) e^{u^\alpha /z}  .
\end{equation}
where, viewing $\tilde S$ as a linear transformation from
$H^*_{\bC^*}(\bP^4)$ with basis $\{\bar e_\alpha\}$ to
$H^*_{\bC^*}(\bP^4)$ with a flat basis, we write
$\tilde S_{i\alpha}(z):= (H^i,\tilde S(z) \bar e_\alpha)$, and
where, viewing $R$ as a linear transformation written in the
normalized canonical basis, we set
$$
\Psi_{i \bar\beta}  :=  (H^i, \bar e_\beta)^\t,\quad
R_{\bar\alpha \bar\beta}(z):=  (\bar e_\beta , R(z)  \bar e_\alpha)^\t,\quad
R_{i \bar \beta}(z):=  (H^i , R(z)  \bar e_\beta)^\t .
$$

\begin{proposition}
  \label{prop:R1asym}
Let $R(z) = 1+R_1 z+R_2 z^2+\cdots$, and
$$
{{(R_k)_{i}}^{\alpha}} := (H^i,  R_k  \, e^\alpha )^\t .
$$
Then, we have
\begin{equation*}
  \Psi_{0 \bar \alpha} =\,  \frac{1+\q_\alpha}{I_0 \cdot \q_\alpha ^2} \cdot \left(\frac{-t}5\right)^{-\frac 32}
\end{equation*}
where $\q_\alpha =  -5\xi^\alpha q^{\frac{1}{5}}$ and {\small
\begin{equation}\label{R1matrix}
 {{(R_1)}_0}^\alpha = \frac{ 1 -\frac{1}{12} \q_\alpha}{\t\cdot  \q_\alpha },\,
  {(R_2)_0}^\alpha = \frac{1 -\frac{13}{12} \q_\alpha -\frac{287}{288} \q_\alpha^2}{\t^2 \cdot \q_\alpha^{2 }},\,  {(R_3)_0}^\alpha= \frac{ \frac{2}{5}-\frac{293}{60}\q_\alpha +\frac{5347}{ 1440}\q_\alpha^2 +\frac{5039}{10368} \q_\alpha^3}{ \q_\alpha^{3 }} .
%\quad \cdots
\end{equation}}
\end{proposition}
\begin{proof}
  By the following Lemma~\ref{asymptoticofI2}, the functions
  \begin{equation} \label{asymptoticofIeq}
    \tilde I_\alpha(q,z) = e^{u^\alpha /z} I_0   R_{0\bar \alpha}(z)  \quad \forall \alpha
  \end{equation}
  are solutions of the Picard--Fuchs equation
  \begin{equation*}
    \Big(D^5 -  q \prod_{k=1}^5 (5D + kz - \t)\Big)   I (q,z) = 0.
  \end{equation*}

  The proposition then follows from the following asymptotic expansion
  of $\tilde I_\alpha$ (Lemma \ref{asymptoticofI} and
  \ref{asymptoticofI2}).
\end{proof}
\begin{lemma} \label{asymptoticofI}
  There exist constants $C_\alpha$, $c_{1\alpha}, c_{2\alpha}, \dotsc$ such that
\begin{align*}
\tilde I_\alpha(z) = & C_\alpha \cdot \frac{1+ \,  \q_\alpha }{ \q_\alpha^2} \Big(  1 + \frac{ 1+c_{1\alpha} \q_\alpha}{ \q_\alpha } \cdot {\t}^{-1}{z} +  \frac{1+(c_{1\alpha}-1)\q_\alpha+c_{2\alpha} \q_\alpha^2}{\q_\alpha^{2 }}\cdot {\t}^{-2}{z}^2  \\&\quad +  \frac{ \frac{2}{5}+(c_{1\alpha}-\frac{24}{5})\q_\alpha+(c_{2\alpha}-c_{1\alpha}-\frac{14}{5})\q_\alpha^2 +c_{3\alpha} \q_\alpha^3}{ \q_\alpha^{3 }}\cdot {\t}^{-3}{z}^3+\cdots\Big) e^{ u^\alpha/z}
\end{align*}
where $u^\alpha$ satisfies
\begin{equation*}
  q \frac d{dq} u^\alpha =  L_\alpha := \frac{ \frac{\t }{5}\q_\alpha}{1 + \, \q_\alpha}.
\end{equation*}
\end{lemma}
\begin{proof}
  By the arguments in the above proposition, applying the
  Picard--Fuchs equation to \eqref{asymptoticofIeq}, we can see that
  $R_{0\bar \alpha}(z)$ satisfies the following equation
  \begin{equation} \label{picardfuchsforR}
    \Big(  D_\alpha^5 -  q \prod_{k=1}^5 (5D_\alpha +kz-\t) \Big)  I_0  R_{0\bar \alpha}(z) = 0,
  \end{equation}
  where $D_\alpha := D + q\frac{d}{dq} u^\alpha$.
  After writing down this equation carefully, we can first solve
  $q \frac d{dq} u^\alpha$ by looking at the coefficient of $z^0$ of
  \eqref{picardfuchsforR}.
  After that, we can determine the coefficients of $z^k$ in
  $R_{0\bar \alpha}(z)$ one by one.
  The coefficient of $z^1$ in \eqref{picardfuchsforR} determines the
  coefficient of $z^0$ in $R_{0\bar \alpha}(z)$, and hence
  $\Psi_{0\bar\alpha}$ up to a constant $C_\alpha$.
  At each step, we need to introduce a new undetermined constant
  $c_{k\alpha}$.
\end{proof}

We cannot directly fix the constant terms of the $R$-matrix because of
the poles of the entries of the $R$-matrix.
The idea to solve this problem is to consider a more general
equivariant theory first, and then take the limit.
We will do so in the following lemma.

\begin{lemma}
  \label{asymptoticofI2}
  The functions
  $\tilde I_\alpha = e^{u^\alpha /z} I_0 R_{0\bar \alpha}(z)$ are
  solutions of Picard--Fuchs equation
  \begin{equation*}
    \Big(D^5 -  q \prod_{k=1}^5 (5D + kz - \t)\Big)   I (q,z) = 0.
  \end{equation*}
  Furthermore, the constants $C_\alpha$ and $c_{i\alpha}$ for
  $i = 1, 2, 3$ in Lemma~\ref{asymptoticofI} are independent of
  $\alpha$, and have the following values
  \begin{equation*}
    C_\alpha = \left(\frac{-t}5\right)^{-\frac 32}, \qquad
    c_{1\alpha} =-\frac{1}{12} ,\qquad
    c_{2\alpha} =-\frac{287}{288} ,\qquad
    c_{3\alpha} = \frac{5039}{10368}.
  \end{equation*}
\end{lemma}
\begin{proof}
  There is a standard method of fixing the constants of the $R$-matrix
  in equivariant Gromov--Witten theory using an explicit formula for
  the $R$-matrix when the Novikov parameters are sent to zero.
  Unfortunately, it does not directly apply to the twisted theory that
  we are considering since we do not work equivariantly on $\bP^4$,
  and the theory hence becomes non-semisimple at $q = 0$.
  To solve this problem, we first introduce a more general equivariant
  theory, and then take a limit to recover the original theory.

  We introduce the $(\mathbb C^*)^5$-action which acts diagonally on
  the base $\bP^4$, and we denote by $\lambda_i$ the corresponding
  equivariant parameters.
  To simplify the computation, we set $\lambda_i = \xi^i \lambda$.
  We consider the corresponding twisted theory of $\bP^4$, which has
  the $I$-function
  \begin{equation*}
    I(\t,\lambda, q,z) =  z \, \sum_{d\geq 0 } q^d  \frac{\prod_{j=1}^{5d}  (5H + jz -\t)}{ \prod_{k=1}^d ((H+k z)^5 - \lambda^5) },
  \end{equation*}
  which satisfies the Picard--Fuchs equation
  \begin{equation}
    \label{PFeq}
    \Big( D_H^5 -\lambda^5 -  q \prod_{k=1}^5 (5D_H+kz-\t) \Big)  I(\t,\lambda,q,z) = 0.
  \end{equation}
  Similarly to the previous discussion, we introduce the mirror map
  $\tau(q)$, the $S$- and $R$-matrix
  \begin{equation*}
    S(\t,\lambda,z), \qquad R(\t,\lambda,z)
  \end{equation*}
  at point $\tau(q)$, the normalized canonical basis, \dots.
  It is clear that by taking $\lambda \rightarrow 0$ we recover the
  twisted theory considered in this paper.

  One main advantage of the more general twisted theory is that it
  stays semisimple at $q = 0$ because the classical equivariant
  cohomology of $\bP^4$ has a basis of idempotents:
  \begin{equation*}
    e_\alpha|_{q = 0}
    = \frac{\prod_{\beta \neq \alpha} (H - \xi^\beta\lambda)}{\prod_{\beta \neq \alpha} (\xi^\alpha\lambda - \xi^\beta\lambda)}
    = \frac{\prod_{\beta \neq \alpha} (H - \xi^\beta\lambda)}{5 \xi^{4\alpha} \lambda^4}
  \end{equation*}
  Hence, the normalized canonical basis stays well-defined at $q = 0$,
  and we have
  \begin{equation*}
    \bar e_\alpha|_{q = 0}
    = \frac{\prod_{\beta \neq \alpha} (H - \xi^\beta\lambda)}{\sqrt{5(5\lambda^5 - t \xi^{4\alpha} \lambda^4)}},
  \end{equation*}
  and
  \begin{equation}
    \label{eq:Psiq0}
    \Psi_{0\bar\alpha}|_{q = 0}
    = \frac{-t + 5\xi^\alpha \lambda}{\sqrt{5(5\lambda^5 - t \xi^{4\alpha} \lambda^4)}}.
  \end{equation}
  By uniqueness of fundamental solutions of the quantum differential
  equation (see also \cite{Gi01a, CoGi07}), we have
  \begin{equation}
    \label{asympeqS}
    S(\t,\lambda,z) (\Psi(\t,\lambda)^{-1}|_{q = 0}) \Gamma^{-1}(\t,z) C^{-1}(\lambda,z) = \Psi(\t,\lambda)^{-1} R(\t,\lambda,z) e^{U(\t,\lambda)/z}
  \end{equation}
  for constant matrices (with respect to $q$) given by
  \begin{equation} \label{constantR} \begin{split}
      C(\lambda,z) = \,& \diag \big( \big\{ e^{ \sum_{k>0,j\neq i} \frac{B_{2k}}{2k({2k-1})}\frac{z^{2k-1}}{(\lambda_i-\lambda_j)^{2k-1}} } \big\}_{i =0,1,2,3,4} \big),\\
      \Gamma(t,z) = \,&e^{ \sum_{k>0} \frac{B_{2k}}{2k({2k-1})}\frac{z^{2k-1}}{(t-5H)^{2k-1}} ,}\end{split}
  \end{equation}
  Indeed, these constant matrices together give exactly the constant
  term of the $R$-matrix.

  Notice that the $\bar e_\alpha|_{q = 0}$ are eigenvectors for the
  eigenvalues $h_\alpha$ of the classical multiplication by $H$.
  Hence, by the Picard--Fuchs equation \eqref{PFeq} for
  $I(t, \lambda, q, z) = z I_0(q) \cdot S(t,\lambda,z)^* \mathbf 1$,
  the identity \eqref{asympeqS}, and the fact that $\Gamma(t, z)$ and
  $C(\lambda, z)$ are diagonal, we see that
  \begin{equation*}
    \tilde I_\alpha(t,\lambda,q,z) = e^{u^\alpha(q)/z} I_0(q)  R_{0\bar \alpha}(z)
  \end{equation*}
  satisfies the Picard--Fuchs equation
  \begin{equation} \label{PFequationlambda}
    \Big( (D + h_\alpha)^5 -\lambda^5 -  q \prod_{k=1}^5 (5D + 5h_\alpha +kz-\t) \Big)  \tilde I_\alpha(\t,\lambda,q,z) = 0.
  \end{equation}
  At the limit $t = 0$, this proves the first part of the lemma.

  Now we can use a similar method as in the proof of
  Lemma~\ref{asymptoticofI} to compute the $R$-matrix.
  By the definition of $\tilde I_\alpha(t,\lambda,q,z)$, we have the
  following Picard--Fuchs equation
  \begin{equation*}
    \Big( (D+L_\alpha)^5 -\lambda^5 -  q \prod_{k=1}^5 (5 D+ 5 L_\alpha +kz-\t) \Big)  I_0(q) R_{ 0 \bar \alpha}(\t,q,z) =0 .
  \end{equation*}
  where $L_\alpha = h_\alpha + q\frac{d}{dq} u^\alpha$.
  The coefficient of $z^0$ gives us an equation for $L_\alpha$:
  \begin{equation}\label{qL}
    q (5 L_\alpha-t)^5 +(L_\alpha^5-\lambda^5) = 0 .
  \end{equation}
  We can choose the basis $\bar e_\alpha|_{q = 0}$ such that the
  solution $L_\alpha$ satisfies the following
  \begin{equation*}
    L_\alpha = \xi^\alpha \lambda +O(q),\quad \text{for} \quad  \alpha = 0,1,2,3,4,
  \end{equation*}
  and from here we see that $h_\alpha = \xi^\alpha \lambda$.
  Then, the coefficient of $z^1$ of the Picard--Fuchs equation and the
  initial condition \eqref{eq:Psiq0} imply that
  \begin{equation*}
    I_0(q) \cdot \Psi_{0\bar \alpha} = \frac{ 5 L_\alpha-\t }{\sqrt{5(5 \lambda^5 - L_\alpha^4 \t)}}.
  \end{equation*}
  Next, we look at the coefficients of $z^2$, $z^3$ and $z^4$ of the equation. By equation \eqref{qL} we have
  \begin{equation*}
    q\frac{d}{dq} L_\alpha =-{\frac { \left( 5\,L_\alpha-\t \right)  \left( {L_\alpha}^{5}-{\lambda}^{5} \right)
}{5\,{L_\alpha}^{4}\t-25\,{\lambda}^{5}}}.
  \end{equation*}
  Using this relation, we can solve $R_1,R_2,R_3$ inductively as
  rational functions of $L_\alpha$.
  The explicit formulae are: {\tiny
\begin{align*}
{(R_1)_{0}} ^{ \alpha} = &\,\frac{1}{(L_\alpha^4 \t-5 \lambda^5)^3}\bigg(-{\frac {13\,L_\alpha^{12}{\t}^{2}}{12}}+{\frac {{\t}^{3}L_\alpha^{11
}}{5}}+ \Big( -15\,\t\,L_\alpha^{8}+5\,{\t}^{2}L_\alpha^{7}-{\frac {
11\,{\t}^{3}L_\alpha^{6}}{30}} \Big) {\lambda}^{5} + \Big( {\frac {75\,{L_\alpha}
^{4}}{4}}+{\frac {10\,L_\alpha^{3}\t}{3}}-{\frac {3\,{\t}^{2}L_\alpha^
{2}}{2}} \Big) {\lambda}^{10}\bigg)
\\
{(R_2)_{0}} ^{ \alpha} = &\,\frac{1}{(L_\alpha^4 \t-5 \lambda^5)^6}\bigg({\frac {L_\alpha^{22}{\t}^{6}}{25}}-{\frac {37\,L_\alpha^{23}{\t}^{5}
}{60}}+{\frac {313\,L_\alpha^{24}{\t}^{4}}{288}}+ \Big( -{\frac {47\,
L_\alpha^{17}{\t}^{6}}{150}}+{\frac {871\,L_\alpha^{18}{\t}^{5}}{72}}-
{\frac {1181\,L_\alpha^{19}{\t}^{4}}{12}}+{\frac {915\,L_\alpha^{20}{
\t}^{3}}{4}} \Big) {\lambda}^{5}\\&\quad+ \Big( {\frac {517\,L_\alpha^{12}{
\t}^{6}}{1800}}-{\frac {67\,L_\alpha^{13}{\t}^{5}}{3}}+{\frac {
6343\,L_\alpha^{14}{\t}^{4}}{24}}-{\frac {37255\,L_\alpha^{15}{\t}^{3}
}{36}}+{\frac {22875\,L_\alpha^{16}{\t}^{2}}{16}} \Big) {\lambda}^{10}\\&\quad+
 \Big( {\frac {209\,L_\alpha^{8}{\t}^{5}}{20}}-{\frac {3379\,L_\alpha^{9}{
\t}^{4}}{18}}+{\frac {22855\,L_\alpha^{10}{\t}^{3}}{24}}-{\frac {
5625\,L_\alpha^{11}{\t}^{2}}{4}}-{\frac {9375\,L_\alpha^{12}\t}{4}}
 \Big) {\lambda}^{15}\\&\quad+ \Big( {\frac {205\,L_\alpha^{4}{\t}^{4}}{8}}-180\,
L_\alpha^{5}{\t}^{3}+{\frac {4675\,L_\alpha^{6}{\t}^{2}}{72}}+{\frac {
5625\,L_\alpha^{7}\t}{2}} +{\frac {5625\,L_\alpha^{8}}{32}} \Big) {\lambda}^{20}
+ \Big( {\frac {5\,{\t}^{3}}{2}}+25\,L_\alpha{\t}^{2}-{\frac {1375
\,L_\alpha^{2}\t}{2}} \Big) {\lambda}^{25}\bigg)\end{align*}\begin{align*}
{(R_3)_{0}} ^{ \alpha} = &\,\frac{1}{(L_\alpha^4 \t-5 \lambda^5)^9}\bigg(-{\frac {56201\,{\t}^{6}L_\alpha^{36}}{51840}}+{\frac {2089\,{\t
}^{7}L_\alpha^{35}}{1440}}-{\frac {73\,L_\alpha^{34}{\t}^{8}}{300}}+{\frac
{2\,L_\alpha^{33}{\t}^{9}}{625}}+ \Big( -{\frac {163861\,L_\alpha^{32}{
\t}^{5}}{96}}+{\frac {324773\,L_\alpha^{31}{\t}^{6}}{288}}\\&\quad-{
\frac {1876667\,L_\alpha^{30}{\t}^{7}}{8640}}+{\frac {109099\,L_\alpha^{29}
{\t}^{8}}{9000}}+{\frac {161\,L_\alpha^{28}{\t}^{9}}{15000}}
 \Big) {\lambda}^{5}+ \Big( -{\frac {25916335\,L_\alpha^{28}{\t}^{4}}{384
}}+{\frac {10077151\,L_\alpha^{27}{\t}^{5}}{216}}-{\frac {7283627\,{L_\alpha}
^{26}{\t}^{6}}{576}}\\&\quad+{\frac {1339837\,L_\alpha^{25}{\t}^{7}}{900}
}-{\frac {6305489\,L_\alpha^{24}{\t}^{8}}{108000}}-{\frac {2023\,L_\alpha^{
23}{\t}^{9}}{45000}} \Big) {\lambda}^{10}+ \Big( {\frac {22253\,L_\alpha^
{18}{\t}^{9}}{810000}}+{\frac {774907\,L_\alpha^{19}{\t}^{8}}{
9000}}+{\frac {17474431\,L_\alpha^{21}{\t}^{6}}{540}}\\&\quad-{\frac {10554529
\,L_\alpha^{20}{\t}^{7}}{3600}}-{\frac {46053127\,L_\alpha^{22}{\t}^{5
}}{288}}-{\frac {12368975\,L_\alpha^{24}{\t}^{3}}{48}}+{\frac {
17130505\,L_\alpha^{23}{\t}^{4}}{48}} \Big) {\lambda}^{15}+ \Big( -{
\frac {236621\,L_\alpha^{14}{\t}^{8}}{6000}}\\&\quad+{\frac {11464411\,L_\alpha^{15
}{\t}^{7}}{5400}}-{\frac {10970767\,L_\alpha^{16}{\t}^{6}}{360}}+
{\frac {6661633\,L_\alpha^{17}{\t}^{5}}{36}}-{\frac {422734975\,L_\alpha^{
18}{\t}^{4}}{864}}+{\frac {28694675\,L_\alpha^{19}{\t}^{3}}{96}}\\&\quad+
{\frac {72617625\,L_\alpha^{20}{\t}^{2}}{128}} \Big) {\lambda}^{20}+
 \Big( -{\frac {560791\,L_\alpha^{10}{\t}^{7}}{1200}}+{\frac {420693
\,L_\alpha^{11}{\t}^{6}}{40}}-{\frac {34497817\,L_\alpha^{12}{\t}^{5}
}{432}}+{\frac {15774785\,L_\alpha^{13}{\t}^{4}}{72}}+{\frac {8465875
\,L_\alpha^{14}{\t}^{3}}{64}}\\&\quad-{\frac {39623875\,L_\alpha^{15}{\t}^{2}
}{32}}+{\frac {1220625\,L_\alpha^{16}\t}{32}} \Big) {\lambda}^{25}+
 \Big( -{\frac {67837\,L_\alpha^{6}{\t}^{6}}{80}} +{\frac {38689\,L_\alpha^
{7}{\t}^{5}}{4}}-{\frac {1706375\,L_\alpha^{8}{\t}^{4}}{96}}-{
\frac {33813575\,L_\alpha^{9}{\t}^{3}}{162}}\\&\quad+{\frac {152344375\,L_\alpha^{
10}{\t}^{2}}{192}}+{\frac {1704375\,L_\alpha^{11}\t}{16}}-{\frac
{36159375\,L_\alpha^{12}}{128}} \Big) {\lambda}^{30} + \Big( -{\frac {693\,L_\alpha^
{2}{\t}^{5}}{4}}-{\frac {175\,L_\alpha^{3}{\t}^{4}}{2}}+{\frac {
738125\,L_\alpha^{4}{\t}^{3}}{24}}\\&\quad-{\frac {1370375\,L_\alpha^{5}{\t}^{
2}}{12}}-{\frac {4296875\,L_\alpha^{6}\t}{24}}+393750\,L_\alpha^{7}
 \Big) {\lambda}^{35} + \Big( 22500\,L_\alpha\t+{\frac {14375\,{\t}^{2}
}{12}}-{\frac {209375\,L_\alpha^{2}}{2}} \Big) {\lambda}^{40}\bigg)
\end{align*}}
Notice that each time we have a constant to fix, and we fix these
constants by \eqref{constantR}.
To be precise, the constant terms of the above
${(R_k)_{0}} ^{ \alpha} $ are fixed the following initial condition
\begin{align*}
{R(z)_{0}}^{ \alpha} |_{L_\alpha = \xi^\alpha \lambda } =
R(z)_{\bar\alpha \bar \alpha} |_{q=0 }= e^{-\frac{1}{12} (
\frac{1}{\t -5 \xi^\alpha \lambda}+\frac{2}{\xi^\alpha \lambda} ) z +\frac{1}{360}
(\frac{1}{(\t -5 \xi^\alpha \lambda)^3}+\frac{-1}{(\xi^\alpha \lambda)^3}) z^3} +O(z^4)
\end{align*}
where we have used
$$
\sum_{\alpha = 1,2,3,4}  \frac{1}{1-\xi^\alpha} = 2,\quad
\sum_{\alpha = 1,2,3,4}  \frac{1}{(1-\xi^\alpha)^3} =-1 .
$$
Finally, by taking $\lambda =0$ in the above formulae of
$R_k(\t,\lambda)$ (note that in this limit $L_\alpha$ becomes
$\frac{\frac{t}{5}\q_\alpha}{1+\q_\alpha}$ as in
Lemma~\ref{asymptoticofI}), we recover the results of
Lemma~\ref{asymptoticofI}, and obtain the constant terms as well.
\end{proof}

\subsection{$\Psi$-matrix and $R$-matrix: computations of the remaining entries}\label{sec:Rmatrix}

From Proposition~\ref{birkhoff}, we are able to compute the $\Psi$-matrix and $R$-matrix using the asymptotic expansion of Proposition~\ref{prop:R1asym}.

\begin{proposition} \label{psimatrixp}
We have the following formula for  the $\Psi$-matrix:
\begin{align*}
\Psi_{0 \bar \alpha} =\,& \Delta_\alpha^{-\frac{1}{2}}=  \frac{1+\q_\alpha}{I_0 \cdot \q_\alpha ^2}  \cdot \left(\frac{ -\t}{5}\right)^{-\frac{3}{2}} \\
\Psi_{1 \bar \alpha}  =\,& \frac{1}{I_{1,1}}\Big( {L_\alpha-I_{1,1;a} \t} \Big) \frac{1+\q_\alpha}{I_0 \cdot \q_\alpha ^2} \left(\frac{ -\t}{5}\right)^{-\frac{3}{2}}=\det \left( \frac{L_\alpha-I_{1,1;a}\t}{I_{1,1}}\right)  \Psi_{0 \bar \alpha}\\
\Psi_{2 \bar \alpha} =\,&   \det\begin{pmatrix}
\frac{L_\alpha-I_{2,2;a}\t}{I_{2,2}} & -\frac{I_{2,2;b}}{I_{2,2}}\t^2  \\
-1& \frac{L_\alpha-I_{1,1;a}\t}{I_{1,1}} \end{pmatrix} \Psi_{0 \bar \alpha}\\
\Psi_{3 \bar \alpha}
=\,&  \det\begin{pmatrix}
\frac{L_\alpha-I_{3,3;a}\t}{I_{3,3}} & -\frac{I_{3,3;b}}{I_{3,3}}\t^2& -\frac{I_{3,3;c}}{I_{3,3}}\t^3 \\
-1 & \frac{L_\alpha-I_{2,2;a}\t}{I_{2,2}}  & -\frac{I_{2,2;b}}{I_{2,2}}\t^2  \\
&-1& \frac{L_\alpha-I_{1,1;a}\t}{I_{1,1}} \end{pmatrix}\Psi_{0 \bar \alpha}
\end{align*}
\end{proposition}
\begin{proof}
Define
 $
{\Psi_i}^\beta := (H^i,e^\beta )^t $.
Then since
$$
\mathbf 1 = \sum_\alpha e_\alpha,
$$
we have
$$
{\Psi_0}^\beta = 1 =
\Psi_{0 \bar \beta }  \Delta_\beta^{\frac{1}{2}} %= \frac{R_0^\alpha}{I_0}
$$
i.e.\ $ \Delta_\beta^{-\frac{1}{2}} = \Psi_{0 \bar \beta } $.
Recall that $\Psi_{0 \bar \beta } $ was computed in
Proposition~\ref{prop:R1asym}.
Using Proposition~\ref{birkhoff}, we get
{\footnotesize
\begin{align*}
\Psi_{0 \bar \alpha} =\,&  \frac{1+\q_\alpha}{I_0 \cdot \q_\alpha ^2} \cdot \left(\frac{ -\t}{5}\right)^{-\frac{3}{2}} \\
\Psi_{1 \bar \alpha}  =\,&\frac 1{I_{1,1}}\Big( {L_\alpha-I_{1,1;a} \t} \Big) \frac{1+\q_\alpha}{I_0 \cdot \q_\alpha ^2} \cdot \left(\frac{ -\t}{5}\right)^{-\frac{3}{2}}
=\det \left( \frac{L_\alpha-I_{1,1;a}\t}{I_{1,1}}\right)  \Psi_{0 \bar \alpha}\\
\Psi_{2 \bar \alpha} =\,&   \det\begin{pmatrix}
\frac{L_\alpha-I_{2,2;a}\t}{I_{2,2}} & -\frac{I_{2,2;b}}{I_{2,2}}\t^2  \\
-1& \frac{L_\alpha-I_{1,1;a}\t}{I_{1,1}} \end{pmatrix} \Psi_{0 \bar \alpha}\\
\Psi_{3 \bar \alpha}
=\,&   \det\begin{pmatrix}
\frac{L_\alpha-I_{3,3;a}\t}{I_{3,3}} & -\frac{I_{3,3;b}}{I_{3,3}}\t^2& -\frac{I_{3,3;c}}{I_{3,3}}\t^3 \\
-1 & \frac{L_\alpha-I_{2,2;a}\t}{I_{2,2}}  & -\frac{I_{2,2;b}}{I_{2,2}}\t^2  \\
&-1& \frac{L_\alpha-I_{1,1;a}\t}{I_{1,1}} \end{pmatrix}\Psi_{0 \bar \alpha}
\end{align*}}%
The relations
$$ \sum_\alpha \Psi_{j \bar\alpha} \Psi_{k \bar\alpha} = (H^j,H^k)^t$$
provide a consistency check of the constants $C_\alpha$ fixed in
Lemma~\ref{asymptoticofI2}.
\end{proof}

We have already computed $R_{0\alpha}$.
Next we compute the remaining entries of the $R$-matrix.
By applying Proposition~\ref{birkhoff} to
Equation~\eqref{asymptoticofIeq}, we have the following inductive
formula: {\footnotesize
\begin{align*}
{R(z)_1}^\alpha  =\,&  \frac{1}{I_{1,1}} \left(\Delta_\alpha^{\frac{1}{2}} q\frac{d}{dq} \Delta_\alpha^{-\frac{1}{2}} {R(z)_0}^\alpha  + {L_\alpha-I_{1,1;a}\t} {R(z)_0}^\alpha  \right)  \\
{R(z)_2}^\alpha  =\,&  \frac{1}{I_{2,2}} \left(\Delta_\alpha^{\frac{1}{2}} q\frac{d}{dq} \Delta_\alpha^{-\frac{1}{2}} {R(z)_1}^\alpha  + ({L_\alpha-I_{2,2;a}\t}){R(z)_1}^\alpha -I_{2,2;b}\t^2   {R(z)_0}^\alpha \right)\\
{R(z)_3}^\alpha  =\,&  \frac{1}{I_{3,3}} \left(\Delta_\alpha^{\frac{1}{2}} q\frac{d}{dq} \Delta_\alpha^{-\frac{1}{2}} {R(z)_2}^\alpha  + ({L_\alpha-I_{3,3;a}\t}){R(z)_2}^\alpha -I_{3,3;b}\t^2   {R(z)_1}^\alpha-I_{3,3;c}\t^3   {R(z)_0}^\alpha \right)\\
{R(z)_4}^\alpha  =\,&  \frac{1}{I_{4,4}} \Big(\Delta_\alpha^{\frac{1}{2}} q\frac{d}{dq} \Delta_\alpha^{-\frac{1}{2}} {R(z)_3}^\alpha + ({L_\alpha-I_{4,4;a}\t}){R(z)_3}^\alpha -I_{4,4;b}\t^2   {R(z)_2}^\alpha\\
& \qquad \qquad -I_{4,4;c}\t^3   {R(z)_1}^\alpha-I_{4,4;d}\t^3   {R(z)_0}^\alpha \Big)
\end{align*}}%
Together with Proposition~\ref{prop:R1asym} (more precisely Equation~\eqref{R1matrix}),
we can then write down all the necessary entries of the $R$-matrix.

\medskip

We will now compute some very explicit entries of the $\Psi$- and
$R$-matrix.
For this, let us introduce the notation
$$
{R^4}_{\bar \alpha}(z) := (\tilde H_4, R(z) \bar e_\alpha)^\t ,\quad {(R_k)^4}_{\bar \alpha}  := (\tilde H_4, R_k \bar e_\alpha)^\t,\quad {(R_k)^{4 \,\alpha}}   := (\tilde H_4, R_k  e^\alpha)^\t .
$$
By applying Lemma~\ref{Hdual4} to the asymptotic expansion formula \eqref{asymptoticofS}, and by using the result in Lemma \ref{asymptoticofI2}, we obtain
\begin{equation} \label{equationofR4}
\left(zq\frac{d}{dq}+ L_\alpha -\frac{\t}{5} \right)   {R^4}_{\bar \alpha}(z)   + \frac{I_0}{5} R_{0 \bar \alpha } (z)=0 .
\end{equation}
Then, we can use this equation to compute the following entries of the
$\Psi$- and $R$-matrix.

\subsubsection{The entries $\Psi^{4\,\alpha }$}

Note that
$$
\Psi^{4\,\alpha } := (\tilde H_4, e^\alpha) = (R_0)^{4\,\alpha }
$$
since $R_0 $ is the identity matrix in any basis.

Consider the coefficient of $z^0$ of Equation \eqref{equationofR4}.
Since $L_\alpha-\frac{\t}{5} = \frac{-\frac{\t}{5}}{1+\q_\alpha}$, we obtain
$$
{\Psi^4}_{\bar \alpha}  =- \frac{1+\q_\alpha}{-\frac{\t}{5}} \frac{I_0}{5} \Psi_{0\bar \alpha } .
$$
Recall that $
\Delta^{\frac{1}{2}} =  \frac{I_0 \cdot \q_\alpha ^2}{1+\q_\alpha}  \cdot \left(\frac{-\t}{5}\right)^{-\frac{3}{2}}= (\Psi_{0\bar \alpha } )^{-1}
$,
so that we have
$$
{\Psi^{4\,\alpha}}   =\frac{1+\q_\alpha}{{\t} } I_0  ,\quad
{\Psi^4}_{\bar \alpha} = \frac{(1+\q_\alpha)^2}{ \q_\alpha^2 \t} \left(\frac{-\t}{5}\right)^{\frac{3}{2}} .
$$

\subsubsection{The entries $(R_k)^{4\,\alpha }$}

Considering the coefficient of $z^1$ of equation \eqref{equationofR4}, we obtain
$$
q\frac{d}{dq} {\Psi^4}_{\bar \alpha}  +\left( L_\alpha -\frac{\t}{5} \right)  {(R_1)^4}_{\bar \alpha}     + \frac{I_0}{5} (R_1)_{0 \bar \alpha} =0
$$
Since $q\frac{d}{dq} \frac{(1+\q_\alpha)^2} { \q_\alpha ^2}=- \frac{2(1+\q_\alpha)} {5 \q_\alpha ^2}$,
we have
$$
-I_0 \frac{2}{5{\t} }  -\frac{\t}{5} \frac{1}{1+\q_\alpha} {(R_1)^{4\, \alpha }}     + \frac{I_0}{5}{(R_1)_0}^{ \alpha}  =0
$$
Hence, by \eqref{R1matrix}, we have
$$
(R_1)^{4\, \alpha} = \frac{(1+\q_\alpha)(12-25 \q_\alpha)}{12\,\t^2 \cdot \q_\alpha}  I_0.
$$
Furthermore, by considering the coefficients of $z^2$ and $z^3$ of equation \eqref{equationofR4} one after another, we deduce
\begin{align*}
(R_2)^{4\, \alpha} =\,& \frac{(1+\q_\alpha)(288-1176 \q_\alpha+625 \q_\alpha^2)}{288\,\t^3 \cdot \q_\alpha^2}  I_0,\\
(R_3)^{4\, \alpha} =\,& \frac{(1+\q_\alpha)(20736-460512\q_\alpha+338868 \q_\alpha^2+11875 \q_\alpha^3)}{51840\,\t^4 \cdot \q_\alpha^3}  I_0 .
\end{align*}

\subsection{Generators and relations}

Next, we derive some basic relations in order to be able to write down
the closed formula for the $S$ and $R$-matrices in terms of a minimal
number of generators.

\subsubsection{Quantum product and relations between $I$-functions}\label{Ikrelations}

Recall that in the flat basis, the quantum product $\dot \tau *_\tau$ can be written in the following form
$$
\dot \tau *_\tau H^{k-1} = I_{k,k} H^{k}+I_{k,k;a} H^{k-1}\t+I_{k,k;b} H^{k-2}\t^2+\cdots .
$$
By \cite{ZZ08}, the functions $I_{k, k}$ have the following properties
\begin{equation} \label{ZZrelations}
I_{0,0}I_{1,1}\cdots I_{4,4}=(1-5^5q)^{-1},\quad I_{p,p} = I_{4-p,4-p}. %,\quad  I_{p,p}= I_{5+p,5+p}.
\end{equation}
Recall $A$ is the matrix for  $ \dot \tau *$ in the flat basis $\{H^k\}$.
\begin{lemma}\label{lem:minimalpolynomial}
The characteristic polynomial of $A$ is given by
\begin{equation} \label{charpoly}
\det(x-A) = \frac{x^5 - q(5x-\t)^5}{1-5^5 q} .
\end{equation}
In particular, we have
\begin{equation} \label{minimalpolynomial}
(\dot \tau *)^5 -q (5\dot \tau * -\t)^5 =0 .
\end{equation}
\end{lemma}
\begin{proof}
Note that the eigenvalues of the matrix $A \frac{dq}{q}$ are just $d u^\alpha$. Recall that by Lemma~\ref{asymptoticofI},
$$d u^\alpha =  L_\alpha \frac{dq}{q},\quad  L_\alpha = \frac{-t \xi^\alpha q^{\frac{1}{5}}}{1-5 \xi^\alpha q^{\frac{1}{5}}} ,$$
so that the characteristic polynomial of $A$ is
$$
\prod_\alpha (x-L_\alpha) = \frac{x^5 - q(5x-\t)^5}{1-5^5 q}  .
$$
In fact, one can see that the numerator is just the coefficient of $z^0$ of \eqref{picardfuchsforR} with $q \frac d{dq} u^\alpha$ replaced by $x$.
\end{proof}

The characteristic polynomial gives us many relations between the
entries $\{I_{k,k;a},I_{k,k;b},\cdots\}$ in the matrix $A$.
However, these do not cover all relations.
For additional relations, we can use the symmetry of the quantum product
\begin{equation}\label{symmetryofqp}
(\dot \tau *_\tau H^k, H^j)^\t = (\dot \tau *_\tau  H^j, H^k)^\t .
\end{equation}
\begin{example}
By taking $(k,j)=(0,3)$ and $(1,2)$ in \eqref{symmetryofqp}, we obtain the following relations between $I_{k,k;a}$:
$$
I_{1,1} - 5I_{1,1;a} = I_{4,4} - 5I_{4,4;a},\quad
I_{2,2} - 5I_{2,2;a} = I_{3,3} - 5I_{3,3;a}
$$
Furthermore, we have
$$
 -5I_{5,5;a} =  I_{0}-1,\quad  \sum_{k=1}^5 I_{k,k;a} = \t^{-1} \Tr A = -  \frac{5^5 q}{ 1-5^5 q} .
$$
Together with the symmetry of $I_{k,k}$: $I_{3,3}=I_{1,1}$ and $I_{4,4}=I_0$,
we deduce
\begin{equation} \label{Ikka}
 I_{1,1;a} + I_{2,2;a}  =\frac{I_{2,2}-1}{10}  -\frac{\frac{1}{2}\cdot5^5 q}{ 1-5^5 q} .
\end{equation}
\end{example}

In the end of this subsection, we conclude that by using the
characteristic polynomial of $A$ and the symmetry of quantum product,
there are indeed only \textbf{two} independent functions in
$\{I_{k,k;a},I_{k,k;b},\cdots\}$:
\begin{lemma} \label{Ikrelationsex}
Denoting the entries in $A$ by $a_{i,j}$, we have
$$
a_{i,j}  \in  I_{1,1}^{-2} I_{2,2}^{-1} \, \mathbb Q [I_{0},I_{1,1},I_{2,2},L,\PP,\QQ],\quad  \forall i,j=0,1,\cdots,4.
$$
\end{lemma}
\begin{proof}
 Let us list all the relations mentioned in this subsection.

 First, the basic relations (Zagier--Zinger's relations) are
$$
I_{k,k} = I_{4-k,4-k},\quad \text{ for } k=0,1,2,3,4 ,
$$
and
$$
I_{0,0}^2I_{1,1}^2I_{2,2} = X:=\frac{1}{1-5^5q}.
$$

Next, we can use the symmetry \eqref{symmetryofqp} and the characteristic polynomial \eqref{charpoly}.
For the extra $I_{5,5;*}$-functions, we have
\begin{align*}
I_{5,5;a} =& \frac{1}{5}(1-I_{0,0}) ,\quad
 I_{5,5;b} =   \frac{ 1}{625} (25-25 I_{3,3}-125 I_{4,4;a}),\\
 I_{5,5;c}  =& \frac{1}{625} (5-5 I_{2,2}-25 I_{3,3;a}-125 I_{4,4;b}),\\
 I_{5,5;d} =&\frac{1}{625}(1-I_{1,1}-5 I_{2,2;a}-25 I_{3,3;b}-125 I_{4,4;c}),\\
  I_{5,5;e} =& \frac{1}{625}(-I_{1,1;a}-5 I_{2,2;b}-25 I_{3,3;c}-125 I_{4,4;d}+\frac{1}{5}(1-I_{0,0}) ) .
\end{align*}
For the extra $I_{4,4;*}$-functions, we have
\begin{align*}
   I_{4,4;a} =&\frac{1}{5}(I_{0,0}+(5 I_{1,1;a}+5 I_{2,2;a}-I_{2,2}-5 I_{3,3;a})),\\
     I_{4,4;b} =& \frac{1}{5}(-I_{2,2;a}+5 I_{2,2;b}+I_{4,4;a}),\\
I_{4,4;c} =& \frac{1}{5}(-I_{3,3;b}+5 I_{3,3;c}+I_{4,4;b}) .
\end{align*}
For the extra $I_{3,3;*}$, $I_{2,2;*}$-functions, we have
$$
I_{3,3;a} = \frac{1}{5}(I_{1,1}+(5 I_{2,2;a}-I_{2,2})),\quad
  I_{2,2;a} = -I_{1,1;a}-\frac{1}{10}(1-I_{2,2}) -\frac{1}{2}(X-1)
$$
Moreover, for the other extra $I$-functions, we have {\footnotesize
\begin{align*}
%I_{3,3;c} =&\frac{1}{I_{1,1}I_{2,2}}\left({\frac{2}{125}}-\frac{X}{50}+ \left( \frac{X}{5}-{\frac{6}{25}} \right) {I_{1,1;a}}+ \left( -\frac{3}{2}\,X+\frac{6}{5} \right) I_{1,1;a}^{2}-2\,I_{1,1;a}^{3}\right)\\ &+{\frac {1}{250}}-\frac{I_{2,2;b}}{10 I_{2,2}} \left(  I_{2,2 }+5 X-4+20 I_{1,1;a} \right)\\
 I_{3,3;b}  = &  \,\,\,\,{\frac {1}{100\,{ I_{2,2}}}}\big({25\,{L}^{10}-40\,{L}^{5}-200\,{L}^{2}\PP+4\,{I_{1,1}}\,{I_{2,2}}
-{I_{2,2}}^{2}}\big)\\
 I_{3,3;c}=&\,-{\frac {10\,{L}^{5}-{{I_{1,1}}}^{2}{I_{2,2}}}{250\,{I_{1,1}}\,{I_{2,2}}}}- {\frac {{L}^{2} \left( 5\,{L}^{5}+{I_{2,2}} \right) \PP}{{
10\,I_{1,1}}\,{I_{2,2}}}}+{\frac {{L}^{6} \left( 5\,{L}^{5}+{I_{2,2}
}-8 \right) \QQ}{{ 20\,I_{1,1}}\,{I_{2,2}}}}+{\frac {{L}^{2}{\QQ}^{2}}{{
10\,I_{1,1}}}}- {\frac {2\,\QQ\PP{L}^{3}}{{I_{1,1}}\,{I_{2,2}}}}\\
 I_{4,4;d}= &  \,\,\,\,{\frac {10\,{I_{1,1}}\,{L}^{5}-10\,{L}^{5}+2\,{I_{0,0}}\, I_{1,1} ^
{2}{I_{2,2}}- I_{1,1} ^{3}{I_{2,2}}}{1250\,  I_{1,1}  ^{2}{I_{2,2}}}
}+{\frac {{L}^{6} \left( -5\,{I_{1,1}}\,{L}^{5}+{I_{1,1}}\,{I_{2,2}}+8
\,{I_{1,1}}-8 \right) }{100\, {I_{1,1}} ^{2}{I_{2,2}}}}\QQ\\
&+{\frac {{L}^{
2} \left( 25\,{L}^{10}-40\,{L}^{5}+2\,{I_{1,1}}\,{I_{2,2}} \right) }{
100\,{{I_{1,1}}}^{2}{I_{2,2}}}}{\QQ}^{2}  \,+\PP \left( {\frac {{L}^{2}  (  5\,{L}^{5}-{I_{2,2}}  ) }{50\,{I_{1,1}
}\,{I_{2,2}}}}-{\frac {{L}^{3} ( 5 {L}^{5} -2 {I_{1,1}})  }{5\, I_{1,1} ^{2}{I_{2,2}}}}\QQ\right.
\\
 & \quad \left. -{\frac {2\,{L}^{4} }{I_{1,1}^{2}{I_{2,2}}}}{\QQ}^{2}\right)+{\frac {{L}^{4}{\PP}^{2}} {I_{1,1} ^{2}{
I_{2,2}}}}
\end{align*}}%
The lemma follows from a careful examination of all these formulae.
\end{proof}

\subsubsection{Identities between derivatives of basic and extra generators}

\begin{lemma}\label{YamaguchiYauident}
We have the following identities between  the basic generators and their derivatives:
\begin{align}
\Y_2=&\, -3 \X_2-\Y^2-\X^2-\frac{15}{4}\Z_2\\
\X_4 =&\,  -4  \X  \X_3-3  \X_2^2-6  \X^2  \X_2- \X^4-\frac{15}{4} (\Z_2  \X^2+\Z_2  \X_2+\Z_3  \X)\\
&\quad -\frac{23}{24} \Z_4+\frac{29}{9} \Z_1^2 \Z_2-\frac{65}{72} \Z_1 \Z_3-\frac{3}{4} \Z_2^2 \nonumber
\end{align}
\end{lemma}
\begin{proof}
This is just a restatement of the identities in \cite{YY04}.
\end{proof}

%Deduce new relations by using Picard--Fuchs equation?

\begin{lemma} \label{diffequationsforPQ}
We have
the following identities for the extra generators
{\small
\begin{align}
\frac{d^2}{du^2}\QQ  = &-2\,\X{ \frac{d}{du}\QQ }-\big(4\,\X_{{2}}+2\,{\X}^{2} +{\frac {15\,\Z_{{2}}}{4}}\big)\QQ\\
&-\frac{5}{2} \,\big( 2\Z_{
{1}}\X_{{2}}+\Z_{{1}}{\X}^{2}+\Z_{{2}}\X\big)-{\frac {125\,\Z_{{1}}\Z_{{
2}}}{24}}-{\frac {5\,\Z_{{3}}}{24}}  \nonumber
,\\
\frac{d^2}{du^2} \PP  = &-(3\,\X +\Y){ \frac{d}{du} \PP }+\big(2\,\X_{{2}} -{\X}^{2} -2\,\Y\X  +{\Y}^{2}  +{\frac {
15\,\Z_{{2}}}{4}}\big)\PP  \\
& -\frac{1}{2}\, \big( (\frac{d}{du}+\X-\Y)\QQ\big)  \big(  (\frac{d}{du}+\X- \Y)(2\QQ+5 \Z_{{1}}) \big)-{\frac {15\,{\Z_{{1}}}^{2}{\X}^
{2}}{4}}-{\frac {55\,\Z_{{1}}\Z_{{2}}\X}{4}} \nonumber
\\
&+\frac{5}{2}\,\big(\Z_{{2}}(\X_{{2}}+{\X}^{2})+
{\Z_{{1}}}^{2}\X_{{2}} +\Z_{{3}}\X\big) +{\frac {305\,{
\Z_{{2}}}^{2}}{64}}-{\frac {35\,{\Z_{{1}}}^{2}\Z_{{2}}}{8}}+10\,{\Z_{{1}}}
^{4}
. \nonumber
\end{align}}
\end{lemma}
\begin{proof}
Recall that in Lemma~\ref{Hdual4}, we have proved the following identity for  two types of  special $S^*$-matrices
\begin{align}
     S^* (t-5H)(\tilde H_4 )   =& \, \tilde H_4, \label{Hdual4a}\\
    S^* \Big( \frac{1}{2}(t-5H) \Big)(\tilde H_4 )   = &\,    \tilde H_4   \label{Hdual4b}
\end{align}
On the other hand,  we have the explicit formula for
 $S^* \mathbf 1$ :
$$
  S^* (t-5H) \mathbf 1 = S^* \Big( \frac{t-5H}{2} \Big)\mathbf 1= I_0^{-1}
$$
and all the other columns of the special $S^*$-matrices can be
computed by Birkhoff factorization starting from $S^* \mathbf 1$ (see
Proposition~\ref{birkhoff}).
Comparing with \eqref{Hdual4a} and \eqref{Hdual4b}, we will get many
identities between basic generators, extra generators $\PP$, $\QQ$ and
their derivatives.

Let us introduce
$$\PP_k:= \frac{d^{k-2}}{du^{k-2}} \PP,\quad  \QQ_k:=\frac{d^{k-1}}{du^{k-1}} \QQ.
$$
In particular, we have $\QQ=\QQ_1,\,\PP=\PP_2 $.
By definition, the degrees of $\PP_k$ and $\QQ_k$ are both $k$.
By Lemma~\ref{YamaguchiYauident} and \eqref{Hdual4a}, we can solve
$\QQ_4$ and $\PP_4$ as rational functions of the basic generators and
$\PP_2,\PP_3,\QQ_1,\QQ_2,\QQ_3$.
Furthermore by using \eqref{Hdual4b}, we can solve $\QQ_3$ as rational
functions of the basic generators and $\PP_2,\PP_3,\QQ_1,\QQ_2$.
Finally, by applying the formula for $\QQ_3$ to the formula for
$\PP_4$, we also write down $\PP_4$ as rational functions of the basic
generators and the four extra generators $\PP_2,\PP_3,\QQ_1,\QQ_2$.
Since the details of solving these equations are tedious, we omit
them.
The resulting two formulae for $\PP_4$ and $\QQ_3$ are precisely what
we claim.
\end{proof}

\begin{corollary}
  The ring generated by all basic generators, extra generators and
  their derivatives is isomorphic to
  \begin{equation*}
    \bQ[\X , \X_2,\X_3, \Y, \QQ,\tilde \QQ,\PP ,\tilde \PP, L].
  \end{equation*}
\end{corollary}
\begin{proof}
  This follows from the preceding relations and Remark~\ref{Zkpolynomial}.
\end{proof}
\begin{remark}
Notice that here  $
 \bQ[\X , \X_2,\X_3, \Y,\QQ,\tilde \QQ,\PP ,\tilde \PP, L]$ is defined as the ring generated by these nine generators.
The
corollary does not imply that there are no other relations between the
generators.
\end{remark}

\section{Proof of key propositions} \label{proofofproposition}

Recall that the key propositions \ref{g0contp}, \ref{g1cont} and
\ref{g2cont} directly imply the Main Theorem.
In this section, we will finish their proofs.

\subsection{Proof of Proposition \ref{g0contp}}  \label{sec:g0cont}

We want to compute the contributions of \textit{graphs with a genus $0$ quasimap vertex}. Recall
that there are two graphs, with contributions given by the following correlators
\begin{align*}
\Cont'_{\Gamma_{a}^{0}}:= &  {L^2}  \left<\left<  \frac{ \frac{5}{3}H^3 \otimes H^4 \t^{-1}+\frac{5}{3} H^4 \t^{-1}\otimes H^3 +\frac{65}{8} H^4 \t^{-1}\otimes H^4 \t^{-1} }{(\t-5H)(\t-5H-\psi_1) \, (\t-5H)(\t-5H-\psi_2)} \right>\right>_{0,2}^{\t},
\\
\Cont'_{\Gamma_{b}^{0}}:= &  {L^2} \left<\left<  \frac{-\frac{5}{3}H^3+\frac{5}{24}H^4 \t^{-1}}{(\t-5H)(\t-5H-\psi_1)},
\frac{-\frac{5}{3}H^3+\frac{5}{24}H^4 \t^{-1}}{(\t-5H)(\t-5H-\psi_2)} \right>\right>_{0,2}^{\t}
\end{align*}

\begin{definition}
We define the following modified $S$-matrix
$$
\mS^*(\t)(\gamma) :=  \gamma + \sum_j e_j \langle\langle e^j, \frac{\gamma}{\t-5H-\psi} \rangle\rangle^\t_{0,2},
$$
and a modified $V$-matrix
$$
\mV(\t)(\gamma_1\otimes \gamma_2) := \left( \gamma_1 \otimes \gamma_2, \frac 1{2t - 5(1 \otimes H + H \otimes 1)}\right)^t + \left\langle\left\langle \frac{\gamma_1}{t-5H-\psi}, \frac{\gamma_2}{\t-5H-\psi} \right\rangle\right\rangle^\t_{0,2}  .
$$
\end{definition}
By definition, we can write the modified matrix as a specialization of the original matrix, for example
$$
\mS^*(\t)   =   \Res_{z = t - 5H} S^* (z) \frac{1}{\t -5H-z} = S^* (z) |_{z=\t-5H}.
$$
\begin{example}
As pointed out in the proof of Lemma \ref{diffequationsforPQ}, by the definition of the $I$-function for the twisted theory \eqref{Itwist}, we have
$$
\mS^*(\t)(\mathbf 1) = \frac{1}{I_0} .
$$
\end{example}

By definition of  $\mV$, the genus $0$ contribution is just the modified $V$-matrix with  given insertions. Since (see e.g. \cite{Gi98})
$$
V(\bt,z,w)(\gamma_1\otimes \gamma_2) = \frac{1}{z+w}\left( S (\bt,z) \gamma_1,  S (\bt,w) \gamma_2 \right)^\t
$$
the modified $V$-matrix can be computed as follows
\begin{align}
&\quad \mV(\t) ( \gamma_1\otimes\gamma_2 ) \nonumber\\
= &\,\, \left(  \gamma_1\otimes\gamma_2,  \frac{\sum_\alpha \mS^* (e_\alpha) \otimes  \mS^* (e^\alpha)}{2\t-5\, (1\otimes H+   H \otimes 1)} \right)^\t\nonumber\\
= &\,\, \left(  \gamma_1\otimes\gamma_2, \frac{ \sum_{j=0,1,2,3}\mS^* (H^j) \otimes  \mS^* (H^{3-j}) -\frac{t^5}{625}
\mS^* (\tilde H_4) \otimes  \mS^* (\tilde H_4)
}{2\t-5\, (1\otimes H+   H \otimes 1)} \right)^\t \label{VtoS}
\end{align}
where we have used $(\tilde H_4, \tilde H_4)^t=-\frac{625}{t^5}$.

Denote by
$$
\mS_k:= \mS^*(t) H^k,\quad  {\mS_{k}}^{i}:=(\tilde H_i, \mS_k)^\t ,
$$
where $\{\tilde H_i\}$ is the dual basis of $\{H^k\}$.
\begin{lemma}
For  $i=0,1,2,3,4; k =0,1,2,3$, the entries of the modified $S$-matrix satisfy
$$
{\mS_{k}}^{i} \in \frac{\t^{k-i}}{I_0\I_1\cdots\I_k} \mathbb Q[\X,\X_2,\X_3,\Y,  L,\PP,\QQ,\PP_3,\QQ_2, \I_1, \I_2 ]
$$
where $\I_k:=\frac{I_{k,k}}{L}$ for $k=1,2,3$ (notice that $\I_3=\I_1$). Moreover, if we define
$$
\deg \I_k =   1 \quad \text{for $k=1,2,3$ }
$$
and rewrite $L^j$ ($j=2,3,4,7,8,12$) in $\mS_k$ as polynomials of
$\Z_k$ ($k=1,2,3$) (see Remark \ref{Zkpolynomial}), then for all
$i=0,1,\cdots,4 $ and $k=0,1,\cdots,3$
$$
(\t^{i-k}I_0 \I_1\cdots\I_k) \cdot {\mS_{k}}^{i}   \in \mathbb Q[\X,\X_2,\X_3,\Y,  \Z, \Z_2, \Z_3,\PP,\QQ,\tilde \PP ,\tilde \QQ , \I_1, \I_2 ]
$$
are homogeneous polynomials of degree $k$.
\end{lemma} \label{mSmatrix}
\begin{proof}
This lemma is a direct consequence of Proposition \ref{birkhoff} and the fact $\mS_0 = I_0^{-1}$. To be more explicit, by a direct Birkhoff factorization computation we have {\tiny
\begin{align*}
 \mS_1 = \, & \frac{1}{I_0\I_{1}} \Big( \big(L^4+5 \X) \cdot H -(\frac{L^4 }{5} +\QQ+\X\big)\cdot\t \Big) \\
 \mS_2 = \, & \frac{1}{I_0\I_{1}\I_{2}}   \Big(  \big(4 L^3-3 L^8+5 L^4 (\X+ \Y )+25\X\Y-25\X_2\big) \cdot H^2\\
 & \big({\frac {-16 L^3 + 17 {L}^{8}}{10}}+\frac{L^4}{2}\,(\X-4\Y)+5\,{\tilde \QQ} -10\X\Y+10\,{\X_2}-\frac{\I_2}{10}(L^4+5\X) \big)\cdot \t H\\
& {\frac {8L^3-11 {L}^{8}}{50}}+{\frac{{L}^{4} }{10}} \left( -3\X+2\Y
 \right)  -\PP-{\tilde \QQ}+\X\Y-{\X_2}+{\frac{\I_2}{50}}\,{ {
 \left( {L}^{4}+5\,Q+5\X \right)  } }\\
 \mS_3 = \, & \frac{1}{I_0\I_{1}^2\I_{2}}  \Big(
 \Big[ 36 L^{12}-47 L^7+12 L^2+(60 L^8  -55 L^3) \X+25 L^4 (\X^2+ \X_2)+125 (\X^3+3 \X \X_2+\X_3)\Big] \cdot H^3\\
 & +\Big[
 -{\frac {211 {L}^{12}-282 L^7 +67 {L}^{2}}{10}}+(33 {L}^{3}-31 {L}^{8})\X-\frac{5}{2} {L}^{4}{\X}^{2}-15\,{L}^{4}\X_2 -75(\X^3+3 \X \X_2+\X_3)+5\,{
  \tilde \QQ} ({L}^{4}+5 \X)
 \\ &\quad
 +(4L^3 -3 {L}^{8})\QQ+5\,{L}^{4}\QQ(\X+ \Y)
 +25\,\QQ(\X\Y -\X_2) %\\ &\quad
 + \frac{\I_1}{5}\left( 3\,{L}^{8}-4\,{L
}^{3}-5\,{L}^{4}(\X+\Y)+25(\X_2- \X\Y) \right)
 \Big]\cdot \t\, H^2 \\
 &+\Big[
 {\frac {103 {L}^{12}-131 {L}^{7}+31 {L}^{
2}}{25}} +{\frac {(26 {L}^{8}-23 L^3)\X}{5}}+ {L}^{4}(3\X_2-2\X^2) +15(\X^3+3 \X \X_2+\X_3)\\
&\quad  +5{\tilde \PP} +5\tilde \QQ(\QQ-2\X+\frac{L^4}{10}) +\left( L^4+5\X\right)\PP +{\frac {\QQ}{10}(17 {L}^{8}-16L^3+5 L^4 (\X-4\Y) )}   + 10\QQ(\X_2-\X\Y)
 \\&\quad - \frac{\I_1}{50}\left(50 \tilde \QQ+17 L^8-16 L^3+5 L^4 (\X-4 \Y)+100 (\X_2- \X \Y)\right) - \frac{\I_1\I_2}{50}\left( L^4+5\X\right)\Big] \cdot \t^2 H\\
 &+\Big[{\frac {-67 {L}^{12}+{74 {L}^{
7}}-9 {L}^{2}}{250}}+\frac{L^4}{10}(3 {\X}^{2}-2\X_2)-{
\frac {(7 {L}^{8}-L^3)\X}{25}}-(\X^3+3 \X \X_2+\X_3)\\
 & \quad   -\tilde \PP-(\frac{3}{10}
 {L}^{4}+\QQ+\X)\tilde \QQ   -( \QQ+\X+\frac{L^4}{5})\PP-\QQ(\X_2-\X\Y) +{\frac {\QQ}{50}}(8L^3-11 {L}^{8}) +\frac{ \QQ}{10}L^4(2\Y-3\X)
 \\
&\quad+\frac{\I_1}{250} \left(11 L^8-8 L^3+ L^4 (15\X-10 \Y)+50(\X_2- \X \Y)+50 P+50 \tilde \QQ \right) +\frac{\I_1\I_2}{250}\left(L^4+5\QQ+5\X\right)
-{\frac {\I_1^2\I_2}{250 }} \Big] \cdot \t^3
  \Big)
\end{align*}
}%
where we have used the relations between $\{I_{k,k;a}, I_{k,k;b} , \cdots \}$ presented in the proof of Lemma \ref{Ikrelationsex}  and the differential equations in Lemma \ref{diffequationsforPQ}
\end{proof}
Finally, by applying the formulae in the proof of Lemma \ref{mSmatrix} and the formula $$\mS^*(\t)(\tilde H_4)=\tilde H_4$$  in Lemma \ref{Hdual4} to
    equation \eqref{VtoS}, with
\begin{align*}
  &  \gamma_1\otimes \gamma_2 =\,\,
 \t^{-2}\,  \frac{ \frac{5}{3}H^3 \otimes H^4 \t^{-1}+\frac{5}{3} H^4 \otimes H^3 \t^{-1} +\frac{65}{8} H^4 \otimes H^4 \t^{-2} }{  1-\t^{-1}\otimes 5H -5H \otimes \t^{-1}+ 25 H\otimes H }  \\
 \text{or} \quad &  \quad \gamma_1\otimes \gamma_2 =\,\,
   \frac{-\frac{5}{3}H^3+\frac{5}{24}H^4 \t^{-1}}{ \t-5H  }\otimes
\frac{-\frac{5}{3}H^3+\frac{5}{24}H^4 \t^{-1}}{ \t-5H },
\end{align*}
we can write down the explicit formula of $L^2 \cdot \mV(t)$ as a homogeneous polynomial of degree $3$ in
$$\mathbb Q[\X,\X_2,\X_3,\Y,  \Z, \Z_2, \Z_3,\PP,\QQ,\tilde \PP ,\tilde \QQ] .$$
Note that here we have used
$$
I_0^2 \I_1^2 \I_2 = L^{2},\quad \deg L^2 =3 .
$$
The explicit computations exactly give us the formulae presented  in Proposition \ref{g0contp}.

\subsection{Proof of Proposition \ref{g1cont}}

There is only one \textit{graph with a genus $1$ quasimap vertex}, and
it has the contribution
$$
\Cont'_{\Gamma^{1}} := \frac{L^2}{I_0}  \left<\left<  \frac{-\frac{5}{3}H^3+\frac{5}{24}H^4 \t^{-1}}{(\t-5H)(\t-5H-\psi)} \right>\right>_{1,1}^{\t}.
$$
This type of twisted invariants can be computed by the following two
steps.
First, by using the modified $S$-matrix $\mS(\t)$, we can write the
modified descendent correlators in terms of the $\mS$ action on the
ancestor genus one invariants (see \cite{KoMa97} and \cite{Gi01b}).
Explicitly, for any
$\alpha \in H^*_{\mathbb C^*}(\bP^4)$, we have
\begin{equation*}
  \langle\langle \frac{\alpha}{\t -5H-\psi}\rangle\rangle_{1,1}^t
  = \langle\langle  \big[\frac {S(\bar \psi) \alpha}{\t-5H-\bar \psi}\big]_+ \rangle\rangle_{1,1}^t
  = \langle\langle  \frac {\mS(\t)\alpha}{\t-5H-\bar \psi}\rangle\rangle_{1,1}^t ,
\end{equation*}
where $\bar \psi$ is the pullback psi class from $\M_{1,1}$, we expand
$\frac{1}{\t-5H-\bar \psi}$ in terms of positive powers of $H$ and
$\psi$ with $S$ acting on the left, and the bracket $[\cdot]_+$ picks
out the terms with a non-negative power of $\psi$.

Next, we evaluate the ancestor invariants of the following form
$$
\langle\langle\gamma(\bar \psi)\rangle\rangle_{1,1}^t  = \int_{\M_{1,1}}  \Omega_{1,1}(\gamma(\psi_1))
$$
by using Givental--Teleman's reconstruction theorem \cite{Gi01b,
  Te12}.
We recall here only the general shape of the reconstruction theorem
and refer the reader to \cite{PPZ15} for more details.
For $2g-2+n>0$, the theorem says that
$$
\Omega_{g,n}(\gamma_1,\cdots,\gamma_n) = RT\omega_{g,n}(\gamma_1,\cdots,\gamma_n) = \sum_{\Gamma \in G_{g,n}} \frac{1}{|\Aut (\Gamma)|}  (\iota_\Gamma)_*T\omega_\Gamma(\gamma_1,\cdots,\gamma_n)
$$
where $G_{g,n}$ is the set of genus $g$, $n$ marked point stable
graphs, $\iota_\Gamma$ is the canonical morphism
$\iota_\Gamma\colon \M_\Gamma \rightarrow \M_{g,n}$, the translation
action $T$ is given by
$$
T\omega_{g,n}(-):= \sum_{k = 0}^\infty  \frac{1}{k!} (p_k)_*  \omega_{g,n+k}(-,T(\psi)^k)
$$
for the cohomological valued formal series
\begin{equation} \label{Tz}
  T(z)= T_1 z^2 +T_2 z^3+ T_3 z^4 +\cdots := z(\mathbf 1 - R^{-1}(z)\mathbf 1) ,
\end{equation}
and below we will define $\omega_\Gamma(\gamma_1,\cdots,\gamma_n)$ for
each graph under discussion.

In our case,  the only insertion $\gamma $ is given by
$$
\gamma(\psi) : = \mS(\t) \frac{-\frac{5}{3}H^3+\frac{5}{24}H^4 \t^{-1}}{(\t-5H)(\t-5H-\psi)} .
$$
The contribution can be written as a summation over two stable graphs:
$$
\Cont'_{\Gamma^{1}}=  \frac{L^2}{I_0} \int_{\M_{1,1}}  \Omega_{1,1}(\gamma(\psi)) =  \frac{L^2}{I_0} \left(\Cont_{\Gamma_a^{1}}+\Cont_{\Gamma_b^{1}} \right)
 $$
 where $\Gamma_a^{1}$ and $\Gamma_b^{1}$ are the following stable graphs
 $$\qquad  \Gamma_a^{1}:=\xy
 (9.5,0); (20,0), **@{-}; (20.2,-0.2)*+{\bullet};
 (21,-2)*+{\displaystyle{{}_{g=1}}};
(8,-0.2)* +{\gamma};
\endxy \quad  \quad \Gamma_b^{1}:= \xy
 (9.5,0); (20,0), **@{-}; (19.7,-0.2)*+{\bullet};  (25,-2)*+{\displaystyle{{}_{g=0}}};
(24.6,0)*++++[o][F-]{}  ; (8,-0.2)*+{\gamma};
\endxy .
$$
We now compute the contribution of each graph separately.

\subsubsection{Contribution of $\Gamma_a^{1}$}

We have
\begin{align*}
 T\omega_{\Gamma_a^1}( \gamma ): =\,&\,\,T\omega_{1,1}(R^{-1}(\psi) \cdot \gamma )\\
= \, &\,\, \omega_{1,1}(\gamma)-\psi_1\, \omega_{1,1}(R_1 \gamma)+\psi_1\, \omega_{1,2}(\gamma, T_1)\\
= \, & \,\,\sum_\alpha (e^\alpha, \gamma)^\t+\psi_1\,  \big(- (e^\alpha, R_1 \gamma)^\t+ (e^\alpha,\gamma)^\t(e^\alpha, T_1)^\t\big)
\end{align*}
where we have used
$$
\omega_{g,n} ( {\gamma_1},\cdots, {\gamma_n})= \Delta^{g-1} \prod_{i = 1}^n (e^\alpha,\gamma_i)^\t .
$$
In particular, we have
\begin{align*}
 T\omega_{\Gamma_a^1}( e_\alpha)
= \, & 1+  \psi_1\,  \big(\- ({R_1^* e_\alpha} , \sum  e^\beta)^\t +
 {(R_1)_0}^\alpha \big).
\end{align*}
Note that the formula for ${(R_1)_0}^\alpha$ is in equation \eqref{R1matrix}.
After a direct computation using the formula for ${(R_1)_0}^\alpha$ and the inductive formulae for the $\Psi$- and $R$-matrix of Section \ref{sec:Rmatrix},  we obtain
\begin{equation} \label{eq:R_a^b}
\begin{split}
 ({R_1^* e_\alpha} , \sum  e^\beta)^\t=&  \,\frac{5^5}{12\t} + \frac{5}{ \t \,\q_\alpha} +\frac{25}{2\t\,\q_\alpha^2 L ^2}(12 (L^5-1)L^2+60 L^3 \X +75 L^4 \tilde \QQ+250 \tilde \PP)\\\qquad
  & +\frac{625}{ \t\,\q_\alpha^3 L^3}( 2 (L^5-1)L^3 + L^4 (8\X+ \Y) +50 L \tilde \PP+(15 L^5 -5 )\tilde \QQ)\\
  & +\frac{625}{ 2\t\,\q_\alpha^4 L^4}(8 (L^5-1)L^4+10 L^5 (2 \X+  \Y)\\&\qquad\qquad\quad   +10 \X+250 L^2 \tilde \PP+25L (3 L^5 -2)   \tilde Q )
  \end{split}
 \end{equation}

\subsubsection{Contribution of $\Gamma_b^{1}$}

Let
$$
W(w,z):= \frac{\sum_\alpha e_\alpha \otimes e^\alpha -R^{-1}(z) e_\alpha \otimes  R^{-1}(w) e^\alpha}{w+z}
$$
then we have
$W(w,z) =W_0+W_1+W_2+\cdots$ where {\small
\begin{align*}
W_0=&\sum_\alpha  R_1 e_\alpha \otimes e^\alpha,\quad
 W_1= \sum_\alpha \Big(-R_2 z- R_2^*w  \Big)e_\alpha \otimes e^\alpha\\
W_2=&\sum_\alpha\Big(R_3z^2+ (R_1R_2-R_3) wz +R_3^*w^2\Big)e_\alpha \otimes e^\alpha
\end{align*}}%
The contribution from the loop type graph is given by
\begin{align*}
 T\omega_{\Gamma_b^1}(\gamma) :
 = \,\,  & T\omega_{0,3}(W(\psi) ,R^{-1}(\psi) \cdot \gamma )\\
= \,\, & \omega_{0,3}(W_0 , \gamma ) \,\,
= \,\,   \sum_\alpha \omega_{0,3}(e_\alpha, R_1e^\alpha, \gamma) \\
=\, \,& \sum_\alpha  \Delta_\alpha^{-1}(e^\alpha, R_1 e^\alpha )^\t  (e^\alpha,\gamma)^\t
 \,\, = \,\, \sum_\alpha  (e^\alpha,\gamma)^\t  (R_1)_{\bar\alpha \bar \alpha}
\end{align*}
In particular,
$$
 T\omega_{\Gamma_b^1}(e_\alpha) =  (R_1)_{\bar\alpha \bar \alpha}.
$$
Again after a direct computation using the formula for ${(R_1)_0}^\alpha$ and using the formulae from Section \ref{sec:Rmatrix}, we obtain
\begin{equation} \label{eq:Raa}
\begin{split}
(R_1)_{\bar\alpha \bar \alpha}  =&  \,\frac{5^5}{12\t} + \frac{5}{ \t \,\q_\alpha} +\frac{25}{2\t\,\q_\alpha^2 L ^2}(4 (L^5-1)L^2+20 L^3 \X +25 L^4 \tilde \QQ+50 \tilde \PP)\\
  & +\frac{125}{ \t\,\q_\alpha^3 L^3}( 2 (L^5-1)L^3 + 10 L^4  \X +50 L \tilde \PP+(25 L^5 -10)\tilde \QQ)\\
  & +\frac{625}{ 2\t\,\q_\alpha^4 L^4}(6 \X+2\Y+ 50 L^2 \tilde \PP+ 5L (4 L^5 -3)   \tilde Q )
  \end{split}
  \end{equation}

To summarize, given
$$
\gamma(\psi ) = \gamma_0+ \gamma_1\psi+\cdots
$$
we have
\begin{align*}
  \Cont_{\Gamma_a^{1}}   = \,&  \int_{\M_{1,1}} T\omega_{\Gamma_a^1}(\gamma) =  \frac{1}{24}   \sum_\alpha \left((\gamma_1,e^\alpha)^\t+  (\gamma_0, e^\alpha)^\t \Big(\-  \big({R_1^* e_\alpha} , \sum  e^\beta\big)^\t +
 {(R_1)_0}^\alpha \Big) \right) ,
   \\
    \Cont_{\Gamma_a^{2}}=  \,& \frac{1}{2}  \int_{\M_{0,3}} T\omega_{\Gamma_a^1}(\gamma) = \frac{1}{2} \sum_\alpha (\gamma_0, e^\alpha)^\t (R_1)_{\bar \alpha \bar \alpha} .
\end{align*}
Finally, by equation \eqref{R1matrix},  \eqref{eq:R_a^b} and \eqref{eq:Raa}, together with the explicit formula for the insertion {\footnotesize
\begin{align*}
(\gamma_0, e^\alpha)  =& \,\, \frac{\t \,I_0}{6000 L^2} \Big(5381L^{12}-7274L^7+1893L^2+L^8(8080\X+410 \Y) +25L^4(80\X_2-43 \X^2
+82 \X\Y)
\\& -8850L^3\X+20250(\X^3+3\X\X_2+ \X_3)-(1975L^4+10250\X)\tilde \QQ+250 \tilde \PP \Big)
-\frac{\t  \,\q_\alpha I_0}{15000}  \Big(3 L^5+\mathcal K L^2\\
&   +15 L \X+38\Big)-\frac{\t \,\q_\alpha^2 I_0}{75000}  \Big(2 L^5+\mathcal K L^2+10 L \X+39\Big)
-\frac{\t \, \q_\alpha^3L I_0}{375000}  \Big (L^4+\mathcal K L+5 \X\Big)-\frac{\t \,\q_\alpha^4 L^2 \mathcal K }{1875000} ,\\
(\gamma_1, e^\alpha)  = & \,\, \frac{ I_0}{6000 L^2} \Big( (2941 L^{12}-4314 L^7+1373 L^2)+(4880  \X+10 \Y) L^8+(1925  \X^2+50  \X \Y+2000  \X_2) L^4\\
&-6850 L^3  \X +10250  (\X^3+3\X_2 \X + \X_3) -\tilde \QQ (250 \X+ 2975 L^4)-9750 \tilde \PP  \Big) + \sum_{k=1}^4 (*)\,\, \q_\alpha^k
\end{align*}}%
where
$\mathcal K:=123 L^8-164 L^3-205 L^4 ( \X+ \Y)+1025 (\X_2- \X \Y)+25
\tilde \QQ$, and where $(*)$ denote some $\alpha$-independent
functions which we have omitted because they are irrelevant for the
further computations, we arrive at the following formulae.
\begin{lemma}
  Both $ \frac{L^2}{I_0} \Cont _{\Gamma_{a}^{1}}$ and
  $ \frac{L^2}{I_0} \Cont _{\Gamma_{b}^{1}}$ are degree $3$ homogeneous polynomials in the
  basic and extra generators, to be precise
  \begin{align*}
   \frac{L^2}{I_0}   \Cont_{\Gamma_a^{1}}=&\,\,\,\,\,\frac {5}{576} \big(\X +2 \Z_{{1}} \big)\tilde \QQ +{\frac {205}{576}}\big(\X_{{3}}+4\X\X
                           _{{2}}+{\X}^{3}-\Y{\X}^{2}-3\Z_{{1}}\Y\X\big)+{\frac {305\,\Z_{{1}}\X_{{2}}}{288}} \\
                         &  +{\frac {5\,\Z_{{1}}{\X}^{2}}{384}} +{\frac {3055\,\Z_{{2}}\X}{2304}}-{\frac {
                           95\,{\Z_{{1}}}^{2}\X}{144}}-{\frac {205\,{\Z_{{1}}}^{2}\Y}{288}}+{\frac {
                           9275\,\Z_{{1}}\Z_{{2}}}{13824}}+{\frac {5285\,\Z_{{3}}}{13824}}
    \\
  \frac{L^2}{I_0}   \Cont_{\Gamma_b^{1}}=&-{\frac {473 }{576}}\tilde \PP+\big({\frac {145\X }{576}}-\frac{\Y}{48} -{\frac {2113\Z_{{
                           1}} }{1152}}\big)\tilde \QQ-{\frac {45}{64}\,(\X_{{3}}+{\X}^{3})}-{\frac {299\,\X\X_{{2}}}{64}}\\
                         &+{
                           \frac {41\,\Y }{48}}(3\X^2-\X_{{2}}+\X\Y)-{\frac {665\,\Z_{{1}}\X_{{2}}}{144}}  +{\frac {2927\,\Z_{{1}}{\X}^
                           {2}}{1152}} +{\frac {4223\,\Z_{{1}}\Y\X}{576}}-
                           {\frac {621\,\Z_{{2}}\X}{256}}\\
                         &-{\frac {659\,{\Z_{{1}}}^{2}\X}{576}}+{
                           \frac {41}{48}}( \Z_{{1}}{\Y}^{2}-\Z_{{2}}\Y)+{\frac {
                           2747\,{\Z_{{1}}}^{2}\Y}{576}}-{\frac {29465\,\Z_{{1}}\Z_{{2}}}{4608}}-{
                           \frac {1555\,\Z_{{3}}}{4608}}
  \end{align*}
\end{lemma}
The proposition is a direct consequence of the above lemma.

\subsection{Proof of Proposition \ref{g2cont}}

Finally, we deal with the remaining \textit{graph with a single genus
  two quasimap vertex}.
It has the contribution
$$
\Cont'_{\Gamma^{2}} := \frac{L^2}{I_0^2} \cdot
\left<\left<   \right>\right>_{2,0}^{\t} .
$$
Again using Givental--Teleman's theorem, we see that we need to compute
$$
\left<\left<   \right>\right>_{2,0}^{\t}  = \int_{\M_{2,0}} \Omega_{2,0}
= \sum_{\Gamma \in G_{2,0}} \frac{1}{|\Aut (\Gamma)|} \Cont_\Gamma.
$$
There are $7$ stable graphs in $G_{2,0}$:
$$
\Gamma^2_1:=\xy
 (19.8,-0.2)*+{\bullet};
 (21,-2)*+{\displaystyle{{}_{g=2}}}\endxy
\qquad
 \Gamma^2_2:=\xy
 (9.5,0); (20,0), **@{-}; (19.8,-0.2)*+{\bullet};
(11,-2)*+{\displaystyle{{}_{g=1}}}; (21,-2)*+{\displaystyle{{}_{g=1}}};
(10.2,-0.2)*+{\bullet};
\endxy
\qquad
\Gamma^2_3:=\xy
  (19.8,-0.2)*+{\bullet};
(16,-2)*+{\displaystyle{{}_{g=1}}};
(24.6,0)*++++[o][F-]{}
\endxy \qquad
\Gamma^2_4:=\xy
 (9.5,0); (20,0), **@{-}; (19.8,-0.2)*+{\bullet};
(11,-2)*+{\displaystyle{{}_{g=1}}}; (25,-2)*+{\displaystyle{{}_{g=0}}};
(24.6,0)*++++[o][F-]{}  ; (10.2,-0.2)*+{\bullet};
\endxy
$$
$$
\Gamma^2_5:= \xy
  (20,-0.2)*+{\bullet};
(16,-2)*+{\displaystyle{{}_{g=0}}};
(25,0)*++++[o][F-]{}  ;(15.2,0)*++++[o][F-]{};
\endxy
\qquad\Gamma^2_6:=\xy
 (9.5,0); (20,0), **@{-}; (19.8,-0.2)*+{\bullet};
(16,-2)*+{\displaystyle{{}_{g=0}}};
(24.6,0)*++++[o][F-]{}  ;(5.6,0)*++++[o][F-]{};(10.4,-0.2)*+{\bullet};
\endxy
\qquad
\Gamma^2_7:=\xy
 (9,0); (18,0), **@{-}; (18.8,-0.2)*+{\bullet};
(21,-2)*+{\displaystyle{{}_{g=0}}};
(14,0)*++++[o][F-]{}  ; (9.2,-0.2)*+{\bullet};
\endxy
$$
Hence we can write the contribution as summation over contributions of the stable graphs:
\begin{equation}
  \label{eq:g2graphsum}
  \Cont'_{\Gamma^{2}}=  \frac{L^2}{I_0^2}\left( {\Cont_{\Gamma^2_1}}+
\frac{{\Cont_{\Gamma^2_2}}}{2}+\frac{{\Cont_{\Gamma^2_3}}}{2}+\frac{{\Cont_{\Gamma^2_4}}}{2}+
\frac{{\Cont_{\Gamma^2_5}}}{8}+
\frac{{\Cont_{\Gamma^2_6}}}{8}+\frac{{\Cont_{\Gamma^2_7}}}{12} \right)
\end{equation}
In the remaining part of this section, we compute each of these contributions.

\subsubsection{Trivial graph}

The first graph $\Gamma^2_1$ is a single genus-two vertex.
We directly compute its contribution:
{\small
\begin{align*}
& \quad \Cont_{\Gamma^2_1} = \int_{\M_{2,0}} T\omega_{2,0}  \\
 = \,& \int_{\M_{2,0}} \left(\omega_{2,0}+(\pi_1)_*\omega_{2,1}(T(\psi)) + \frac{1}{2!}(\pi_2)_*\omega_{2,2}(T(\psi)^2)+ \frac{1}{3!}(\pi_3)_*\omega_{2,3}(T(\psi)^3)\right)\\
=\,& \quad  \omega_{2,1}(T_3) \int_{\M_{2,1}} \psi_1^4 + \frac{2}{2}\omega_{2,2}(T_1,T_2) \int_{\M_{2, 2}} \psi_1^2\psi_2^3
+ \frac{1}{6}\omega_{2,3}(T_1^3) \int_{\M_{2, 3}} \psi_1^2\psi_2^2\psi_3^2 \\
=\,& \sum_\beta\Delta_\beta \Big(\frac{1}{1152}  (e^\beta,T_3)^\t  +  \frac{ 29}{5760}  (e^\beta,T_1)^\t (e^\beta,T_2)^\t
+ \frac{7}{6\cdot 240} \left((e^\beta,T_1)^\t \right)^3   \Big)
\end{align*}}%
where we have used some known intersection numbers in $\M_{g,n}$.

\subsubsection{One-edge graphs}

There are two graphs with one edge:
$$
\Gamma^2_2:=\xy
 (9.5,0); (20,0), **@{-}; (20.2,-0.2)*+{\bullet};
(11,-2)*+{\displaystyle{{}_{g=1}}}; (21,-2)*+{\displaystyle{{}_{g=1}}};
(10.2,-0.2)*+{\bullet};
\endxy \qquad
\Gamma^2_3:=\xy
  (19.8,-0.2)*+{\bullet};
(16,-2)*+{\displaystyle{{}_{g=1}}};
(24.6,0)*++++[o][F-]{}
\endxy
$$
For the computation of $\Cont_{\Gamma^2_2}$, it is useful to notice
that
 \begin{align*}
   T\omega^\tau _{1,1}(  e_\beta)
 =\,&   1 +  (e^\beta,T_1)^\t  \cdot \psi  .
\end{align*}
With this, we can compute:
{\small
 \begin{align*}
  \textrm{Cont}_{\Gamma^2_2} =\, &\int_{\M_{1,1}\times \M_{1,1}} (T\omega_{1,1}\otimes  T\omega_{1,1}) (W(\psi_1,\psi_2)) \\
 =\,& \frac{1}{24^2} \sum_{\alpha,\beta}  \Big( (e^\beta,T_1)^\t (e^\alpha,R_1e^\beta)^\t (e^\alpha,T_1)^\t -2
      (e^\alpha,( R_2^*) e^\beta)^\t (e^\beta,T_1)^\t\\
   &\qquad \qquad  +    (e^\alpha,(R_1R_2-R_3) e^\beta )^\t \Big)
 \end{align*}}%

We next compute $\Cont'(\Gamma^2_3)$ using
\begin{align*}
 T\omega_{1,2}(-) = \,&   \omega_{1,2}(-)+(\pi_1)_*\omega_{1,3}(-,T(\psi)) + \frac{1}{2!} (\pi_2)_*\omega_{1,4}(-,T(\psi)^2)
\end{align*}
and $R_2+R_2^*=R_1^2$: {\small
\begin{align*}
\Cont_{\Gamma^2_3} =\, &  \int_{\M_{1,2}} T\omega_{1,2}(W(\psi_1,\psi_2)) \\
=\,&  \left< \omega_{1,2}(W_2(\psi_1,\psi_2))\right>_{1,2}+ \left< \omega_{1,3} (W_1(\psi_1,\psi_2),T_1),\psi_3^2\right>_{1,3}\\
&\qquad+\omega_{1,3}(W_0,T_2)\left< \psi_3^3\right>_{1,3}
 + \frac{1}{2!} \omega_{1,4}(W_0,T_1^2)\left<\psi_3^2\psi_4^2\right>_{1,4}
 \\=\,& \sum_\beta \Delta_\beta\Big(\frac{1}{24}  (e^\beta, (R_3^* +R_1R_2  )e_\beta)^\t  - \frac{1}{12} (e^\beta,(R_2+R_2^*)e_\beta)^\t  (e^\beta,T_1)^\t  \\
&\qquad+\frac{1}{24} (e^\beta, R_1 e_\beta)^\t (e^\beta,T_2)^\t
 + \frac{1}{2\cdot 6} (e^\beta, R_1 e_\beta)^\t \left((e^\beta,T_1)^\t\right)^2 \Big)
\end{align*}}

\subsubsection{Two-edge graphs}

There are two graphs with two edges:
{$$
\Gamma^2_4:=\xy
 (9.5,0); (20,0), **@{-}; (19.8,-0.2)*+{\bullet};
(11,-2)*+{\displaystyle{{}_{g=1}}}; (25,-2)*+{\displaystyle{{}_{g=0}}};
(24.6,0)*++++[o][F-]{}  ; (10.2,-0.2)*+{\bullet};
\endxy\qquad
\Gamma^2_5:=\xy
  (20.2,-0.2)*+{\bullet};
(16,-2)*+{\displaystyle{{}_{g=0}}};
(25,0)*++++[o][F-]{}  ;(15.2,0)*++++[o][F-]{};
\endxy
$$}
We compute $\Gamma^2_4$ as follows:
{\small
 \begin{align*}
    \Cont_{\Gamma^2_4}=\,&\int_{\M_{1,1}} (T\omega_{1,1} \otimes \omega_{0,3}) (W_0+W_1,W_0) \\
 =\,& \frac{1}{24}\Big( \sum_{\alpha,\beta}  (e^\beta,T_1)^\t (e^\beta,R_1 e_\alpha)^\t (e^\alpha, R_1 e^\alpha)^\t
 +  (e^\beta,-R_2^* e_\alpha)^\t (e^\alpha, R_1 e^\alpha)^\t \Big) \\
 =\,& \frac{1}{24} \sum_{\alpha,\beta} \Delta_\alpha^{1/2}\Delta_\beta^{1/2} \Big( (e^\beta,T_1) ^\t (R_1)_{\alpha\beta} (R_1)_{\alpha\alpha}
 - (R_1)_{\alpha\alpha}(R_2^*)_{\beta\alpha} \Big)
 \end{align*}}
For $\Gamma^2_5$, recalling
$$
T \omega _{0,4}(-) = \omega _{0,4}( -) +  \omega _{0,5}(-, T_1)\cdot  \psi_1,
$$
we can write {\small
\begin{align*}
   \Cont_{\Gamma^2_5}=\,&\int_{\M_{0,4}} T \omega _{0,4}(W \otimes W)\\
=\,&
\omega_{0,5}(W_0^2,T_1) -2\sum_\alpha \omega_{0,4}(W_0 , (R_2+R_2^*)e_\alpha,e^\alpha)\\
=\,&\sum_\beta \Delta_\beta \Big(  \left((e^\beta,R_1 e_\beta )^\t\right)^2 (e^\beta,T_1)^\t -2  (e^\beta,R_1 e_\beta )^\t (e^\beta,(R_2+R_2^*) e_\beta )^\t\Big)
\end{align*}}

\subsubsection{Three-edge graphs}

There are also two graphs with three edges:
$$\Gamma^2_6:=\xy
 (9.5,0); (20,0), **@{-}; (19.8,-0.2)*+{\bullet};
(16,-2)*+{\displaystyle{{}_{g=0}}};
(24.6,0)*++++[o][F-]{}  ;(5.6,0)*++++[o][F-]{};(10.4,-0.2)*+{\bullet};
\endxy
\qquad
\Gamma^2_7:=\xy
 (9,0); (18,0), **@{-}; (18.7,-0.2)*+{\bullet};
(21,-2)*+{\displaystyle{{}_{g=0}}};
(14,0)*++++[o][F-]{}  ; (9.2,-0.2)*+{\bullet};
\endxy
$$
The computation of their contributions is similar to the previous ones: {\small
\begin{align*}
\Cont_{\Gamma^2_6} = &\, (\omega_{0,3}\otimes \omega_{0,3}) (W_0\otimes W_0\otimes W_0)\\
  = \,& \sum_{\alpha, \beta}  (R_1 e^\beta,e^\beta)^\t (e_\alpha *_\tau R_1 e^\alpha,  R_1 e_\beta)^\t\\
  =\, &   \sum_{\alpha, \beta} (\Delta_{\alpha}  \Delta_{\beta})^{1/2}  (R_1)_{\alpha\alpha} (R_1)_{\beta\beta} (R_1)_{\alpha\beta}
\\
 \Cont_{\Gamma^2_7} =&\,% (\omega_{0,3} \otimes \omega_{0,3}) (W_0\otimes W_0\otimes W_0) =
 \,\sum_{\alpha, \beta}  (\Delta_{\alpha}  \Delta_{\beta})^{1/2} (R_1)_{\alpha\beta}^3
\end{align*}}%

\subsubsection{Conclusion}

For $k=1,2,\cdots, 7$, we introduce
$$
\Cont'_{\Gamma^2_{k}} :=  \frac{L^2}{I_0^2} \Cont_{\Gamma^2_{k}} .
$$
After a long but direct computation using the $R$-matrix formulae in
Section \ref{sec:Rmatrix}, we arrive at the following proposition.
\begin{proposition}
  \label{prop:g2graphs}
Both $\Cont'_{\Gamma^2_{1}}$ and $\Cont'_{\Gamma^2_{2}}$ are degree $3$ homogeneous polynomials in the basic generators: { \small
\begin{align*}
\Cont'_{\Gamma^2_1} = &\,\,\,  \frac{1}{2^{15}3^6} (24475 \Z_3 -17565 \Z_1 \Z_2)\\
\Cont'_{\Gamma^2_2} = &\,\,\, {\frac {5}{576}}\big(\X_{{3}}+ 4 \X\X_{{2}}+{\X}^{3} - \Y{\X}^{2}\big) +{\frac {5}{144}}\big( \Z_{{1}
}\X_{{2}}-\Z_{{1}}\Y\X -  \Z_{
{1}} ^{2}(\X+\Y)\big) \\
 &+{\frac {25}{768}}\,\Z_{{2}}\X  + \frac{1}{2^{14}3^6} \big( 139157 \Z_3-12531 \Z_1 \Z_2\big)
\end{align*}
}%
Both $\Cont'_{\Gamma^2_{3}}$ and $\Cont'_{\Gamma^2_{4}}$ are degree $3$ homogeneous polynomials in the basic and extra generators:
 {\small
 \begin{align*}
\Cont'_{\Gamma^2_3} = & \,\,\, {\frac {289 }{6912}}\tilde \PP+{\frac {1345}{13824}}\,\Z_{{1}} \tilde \QQ \\
&-\frac{1}{24}\,\big(\X_{{3}}+2\X\X_{{2
}}+{\X}^{3}+\Y{\X}^{2}\big)-\frac{1}{6}\,\Z_{{1}}\big({\X}^{2}+ \Y\X\big)\\
&- {
\frac {3347\,\Z_{{2}}\X+390\,{\Z_{{1}}}^{2}\X-682\,\Z_{{2}}\Y-1630\,{\Z_{{1}}}^{2}\Y}{13824}} +\frac{ 21573 \Z_1 \Z_2 -68155\Z_3 }{2^{12}3^{5}}
 \\
\Cont'_{\Gamma^2_4} = & -{\frac {289 }{6912}}\tilde \PP +\big({\frac {5\,\X }{288}}-{\frac {865\,\Z_{{1}} }{
13824}}\big)\tilde \QQ \\
& +\frac{1}{24} (3\X+\Y)(\X\Y-\X_2)  -{\frac {5\,\Z_{{1}}\X_{{2}}}{24}} +{\frac {53\,\Z
_{{1}}{\X}^{2}}{192}} +\frac{1}{12}\,\Z_{{1}}{\Y}^{2}+{\frac {151\,\Z_{{1}}\Y\X}{288}}\\&-{
\frac {205\,\Z_{{2}}\X}{13824}}+{\frac {1055\,{\Z_{{1}}}^{2}\X}{2304}}-{\frac {235\,\Z_{{2}}\Y}{6912}} +{\frac {3455\,{\Z_{{1}
}}^{2}\Y}{6912}}+{\frac {46085\,\Z_{{1}}\Z_{{2}}-6875\,
\Z_{{3}}}{331776}}
\end{align*}
}%
The sum $\Cont'_{\Gamma^2_{5}} + \Cont'_{\Gamma^2_{6}} + \Cont'_{\Gamma^2_{7}}$ is a degree $3$ homogeneous polynomial in the basic and extra generators:
 {\small
 \begin{align*}
\Cont'_{\Gamma^2_{5,6,7}} := & \quad  \frac{1}{8}\Cont'_{\Gamma^2_{5}}+ \frac{1}{8}\Cont'_{\Gamma^2_{6}}+ \frac{1}{12}\Cont'_{\Gamma^2_{7}}
\\
 = & \,
-{\frac {5}{1152}}\tilde \PP-\big(\frac{3\X+\Y}{48} +{\frac {245\Z_{{1}}}{2304}}\big) \tilde \QQ\\
&-\frac{1}{24}\,\big({\X}^{3}+ 3\,\Y{\X}^{2}+3\,{\Y}^{2}\X+{\Y}^{3}+8\,\Z_{
{1}}{\X}^{2}+ 11\,\Z_{{1}}(\Y\X+\frac{\Y^2}{2})+\Z_{{2}}\Y\big)\\
&+{\frac {161\,\Z_{
{2}}\X}{2304}}-{\frac {325\,{\Z_{{1}}}^{2}\X}{384}}   -{\frac {605\,{\Z_{{1}}}^{2}\Y}{
1152}}-{\frac {15449\,\Z_{{1}}\Z_{{2}}}{27648}}+{\frac {5225\,\Z_{{3}}}{
82944}}
\end{align*}
}%
\end{proposition}
Proposition~\ref{g2cont} now follows directly from \eqref{eq:g2graphsum}
and Proposition~\ref{prop:g2graphs}.

\smallskip

For reference, we also provide individual formulae for
$\Cont'_{\Gamma^2_{5}}$, $\Cont'_{\Gamma^2_{6}}$ and
$\Cont'_{\Gamma^2_{7}}$.
They are rational functions in the basic and extra generators.
\begin{lemma}
Introduce the following inhomogeneous polynomial
\begin{align*}
\mathcal E:= \,\,\,& \Big( 4\, \tilde Q ^{2}+2\, \left( 3\,\X+\Y
 \right)  \tilde P  \Big)  \,L-20\, \tilde Q \tilde P     \,{L}^{2}+25\, \tilde P ^{2} \,{L}^{3}\\
 &+\frac{1}{5}\tilde Q
 ( 3\,\X+\Y ) \, {L}^{5}-10\, \tilde Q ^{2} \,
{L}^{6}+25\, \tilde Q  \tilde P
\, {L}^{7}+\frac{1}{25} \left( 3\,\X+\Y \right) ^{2}{L}^{9}+{\frac {25
 }{4}}\tilde Q ^{2}{L}^{11}  .
\end{align*}
We have the following polynomiality for $\Cont'_{\Gamma^2_{5}}$ , $\Cont'_{\Gamma^2_{6}}$ and $\Cont'_{\Gamma^2_{7}}$: {\small
 \begin{align*}
\Cont'_{\Gamma^2_{5}} = & \, - \frac{3 \mathcal E}{L^5-1}+ {\frac {1183 }{90}}\tilde \PP -\big({\frac {79\,\X }{6}}+{\frac {71\,\Y }{10}}-{\frac {
3221\,\Z_{{1}} }{360}}\big) \tilde \QQ  \\
&  +\frac{1}{5}\big(6 \,{\X}^{3}-2 \,\Y{\X}^{2}-8 \,{\Y}^{2}\X\big)+{\frac {101\,
\Z_{{1}}{\X}^{2}}{4}}+{\frac {431\,\Z_{{1}}\Y\X}{30}}+{\frac {43\,\Z_{
{1}}{\Y}^{2}}{20}}\\
&  -{\frac
{119\,\Z_{{2}}\X}{72}}+{\frac {61\,\Z_{{2}}\Y}{36}}+{\frac {879\,{\Z_{{1}}}^{2}\X}{20}}+{\frac {557\,{\Z_{{1}}}^{2
}\Y}{180}}+{\frac {176539\,\Z_{{1}}\Z_{{2}}}{6912}}-{\frac {17389\,\Z_{{3}
}}{6912}}
\\
\Cont'_{\Gamma^2_{6}} = & \,\,\,\, \frac{ \mathcal E}{L^5-1} -{\frac {869}{240}}\tilde \PP+\big({\frac {21\,\X }{5}}+{\frac {11\,\Y }{5}}-{\frac {
439\,\Z_{{1}} }{288}}\big) \tilde \QQ\\
&-\frac{1}{5}\big(7 \,{\X}^{3}+9 \,\Y{\X}^{2}+3
\,{\Y}^{2}\X+{\Y}^{3}\big)-{\frac {877\,\Z_{{1}}{\X}^{2}}{60}}-{\frac {359\,\Z_{{1}}\Y\X}{30}} -{\frac {51\,\Z_{{1}}{\Y}^{
2}}{20}}\\
& -{\frac {33\,\Z_{{2}}\X}{32}}-{\frac {31\,\Z_{{2}}\Y}{72}} -{
\frac {967\,{\Z_{{1}}}^{2}\X}{48}}-{\frac {1087\,{\Z_{{1}}}^{2}\Y}{144}
}-{\frac {76795\,\Z_{{1}}\Z_{{2}}}{6912}}+{\frac {6565\,\Z_{{3}}}{6912}}
\\
\Cont'_{\Gamma^2_{7}} = & \, \,\,\, \frac{3 \mathcal E}{L^5-1}
-{\frac {1147 }{80}}\tilde \PP+\big({\frac {127\,\X }{10}}+{\frac {71\,\Y }{10}}-{
\frac {5957\,\Z_{{1}} }{480}}\big)\tilde \QQ\\
&+\frac{1}{5} \big(-\,{\X}^{3}+9 \,\Y{\X}^{2}+9 \,{\Y}^{2}\X-\,{\Y}^{3}\big)-{
\frac {399\,\Z_{{1}}{\X}^{2}}{20}}-{\frac {91\,\Z_{{1}}\Y\X}{
10}}-{\frac {43\,\Z_{{1}}{\Y}^{2}}{20}}\\
& +{\frac {467\,\Z_{{2}}\X}{96}}-{\frac {67\,\Z_{{2}}\Y}{48}}-{\frac {3669\,{\Z_{{1}}}^{2}\X}{80}}
+{\frac {91\,{\Z_{{1}}}^{2}\Y}{240}}-{\frac {65321\,\Z_{{1}}\Z_{{2}}}{2304
}}+{\frac {21461\,\Z_{{3}}}{6912}}
\end{align*}
}
\end{lemma}

The geometric meaning of this inhomogeneous polynomial is unclear.
It has the following expansion
$$
\frac{\mathcal E}{L^5-1} = 45 q+227400 q^2+1195603370 q^3+5913833272300 q^4+O(q^5) .
$$
All  coefficients are integers.

\bibliographystyle{abbrv}
\bibliography{g2}

\end{document}